\pdfoutput=1
\documentclass[12pt,reqno,final]{amsart}
\usepackage{amsmath,a4wide}
\usepackage{stmaryrd,mathrsfs,bm,amsthm,amssymb, xcolor,tikz}
\usetikzlibrary{arrows,shapes,patterns,calc,fadings,arrows.meta,decorations.pathmorphing,decorations.pathreplacing,decorations.markings}
\usepackage[backend=biber,style=alphabetic,bibencoding=utf8,language=auto,autolang=other,giveninits=true,doi=false,isbn=false,url=false,maxnames=10,sorting=nty]{biblatex}
\usepackage{hyperref}
\hypersetup{
	colorlinks = true,
	linkcolor = {blue},
	urlcolor = {red},
	citecolor = {blue}
}
\usepackage{enumerate}
\usepackage{enumitem}
\usepackage{comment}
\usepackage{listings}
\usepackage[frozencache,cachedir=.]{minted}

\usepackage[nameinlink,capitalise,sort]{cleveref}
\crefname{equation}{}{} 
\crefname{enumi}{}{} 
\crefname{figure}{Figure}{Figures}

\begingroup
\newtheorem{theorem}{Theorem}[section]
\newtheorem{lemma}[theorem]{Lemma}
\newtheorem{proposition}[theorem]{Proposition}
\newtheorem{corollary}[theorem]{Corollary}
\endgroup

\theoremstyle{definition}
\newtheorem{definition}[theorem]{Definition}
\newtheorem{remark}[theorem]{Remark}

\newtheorem{conjecture}[theorem]{Conjecture}

\newcommand{\R}{\ensuremath{\mathbb R}} 
\newcommand{\Z}{\ensuremath{\mathbb Z}} 
\newcommand{\N}{\ensuremath{\mathbb N}} 
\DeclareMathOperator{\diam}{diam}
\DeclareMathOperator{\dist}{dist} 
\newcommand{\eps}{\ensuremath{\varepsilon}}

\newcommand\oneminM{%
\ensuremath{\text{-}\mathrm{One\hspace{1pt}Min\hspace{1pt}M}}%
}

\newcommand\InfminM{\ensuremath{\mathrm{Min\hspace{1pt}M}}}
\newcommand\InfoneminM{\ensuremath{\mathrm{One\hspace{1pt}Min\hspace{1pt}M}}}

\numberwithin{equation}{section}

\newcommand\res{\mathop{\hbox{\vrule height 7pt width .3pt depth 0pt\vrule height .3pt width 5pt depth 0pt}}\nolimits}

\title[Besicovitch's problem and linear programming]{Besicovitch's $\frac{1}{2}$ problem and linear programming}
\author[C. De Lellis]{Camillo De Lellis}
\address{School of Mathematics, Institute for Advanced Study, 1 Einstein Dr., Princeton NJ 08540, USA}
\email{camillo.delellis@ias.edu}
\author[F. Glaudo]{Federico Glaudo}
\address{School of Mathematics, Institute for Advanced Study, 1 Einstein Dr., Princeton NJ 08540, USA}
\email{federico.glaudo@ias.edu}
\author[A. Massaccesi]{Annalisa Massaccesi}
\address{Dipartimento di Matematica T. Levi-Civita, via Trieste 63, 35121 Padova, Italy
\newline\indent
School of Mathematics, Institute for Advanced Study, 1 Einstein Dr., Princeton NJ 08540, USA} 
\email{annalisa.massaccesi@unipd.it}
\author[D. Vittone]{Davide Vittone}
\address{Dipartimento di Matematica T. Levi-Civita, via Trieste 63, 35121 Padova, Italy
\newline\indent
School of Mathematics, Institute for Advanced Study, 1 Einstein Dr., Princeton NJ 08540, USA}
\email{davide.vittone@unipd.it}

\begin{document}

\begin{abstract}
We consider the following classical conjecture of Besicovitch: a $1$-dimensional Borel set in the plane with finite Hausdorff $1$-dimensional measure $\mathcal{H}^1$ which has lower density strictly larger than $\frac{1}{2}$ almost everywhere must be countably rectifiable. We improve the best known bound, due to Preiss and Ti\v{s}er, showing that the statement is indeed true if $\frac{1}{2}$ is replaced by $\frac{7}{10}$ (in fact we improve the Preiss-Ti\v{s}er bound even for the corresponding statement in general metric spaces). More importantly, we propose a family of variational problems to produce the latter and many other similar bounds and we study several properties of them, paving the way for further improvements. 
\end{abstract}

\maketitle

\tableofcontents

\section{Introduction}
The aim of this note is to report on some progress on the following well-known conjecture. 

\begin{conjecture}\label{c:Bes}
Assume $E\subset \mathbb R^2$ is a Borel set with $\mathcal{H}^1 (E) < \infty$ and assume that 
\[
\Theta^1_* (E, x) := \liminf_{r\downarrow 0} \frac{\mathcal{H}^1 (B_r (x)\cap E)}{2r} > \frac{1}{2}
\quad \mbox{for $\mathcal{H}^1$-a.e. $x\in E$.}
\]
Then $E$ is countably $1$-rectifiable. 
\end{conjecture}

As usual, countable $k$-rectifiability means that the set in question can be covered, up to an $\mathcal{H}^k$-null set, by countably many Lipschitz images of $\mathbb R^k$.
\cref{c:Bes} is perhaps the oldest open problem in Geometric Measure Theory and was stated at first in Besicovitch's foundational work about his theory of ``linearly measurable subsets of the plane'', cf. \cite{Besicovitch0}. At page 454 he writes: ``It is a most interesting question to find the exact value of this bound. It is plausible that the right theorem is\footnote{In Besicovitch's terminology ``irregular'' is equivalent to ``purely unrectifiable''.}: 

\smallskip

\noindent {\em At almost all points of an irregular set the lower density is always less than or equal to $\frac{1}{2}$.}''

\smallskip

 The problem is then further mentioned in \cite[Page 329]{Besicovitch} and became well known among experts, see e.g. \cite[Page 44]{Falconer}. 
 It is convenient to reformulate it as follows.

\begin{definition}\label{def:bar-sigma}
Denote by $\bar{\sigma}$ the infimum of all numbers $\sigma$ for which the following statement is true for every Borel set $E\subset \R^2$ with $\mathcal{H}^1 (E)<\infty$.
\begin{itemize}
    \item[(B)] If $\Theta^1_* (E,x) \geq \sigma$ for $\mathcal{H}^1$-a.e. $x\in E$, then $E$ is 
countably $1$-rectifiable.
\end{itemize}
\end{definition}

\cref{c:Bes} states therefore that $\bar \sigma =\frac{1}{2}$ and the main progress of this note can be summarized in the following theorem. 

\begin{theorem}\label{t:best-bound}
$\bar\sigma \leq 0.7$.
\end{theorem}

\begin{figure}[htbp]
\begin{tikzpicture}[scale=0.8,transform shape]
\tikzset{bharrow/.style={
    decoration={markings,mark=at position 1 with {\arrow[scale=2]{>}}},
    postaction={decorate},
    }
}
\begin{scope}[every node/.style = {
    shape = rectangle,
    minimum width= 5cm,
    minimum height=0.8cm,
    align=center}]
    \node (besicovitch) at (0,0) [draw=gray] { \scalebox{1.25}{
        $\bar\sigma \le \sigma$
    }};
    \node (minm) at (4, -2) [draw=gray] {\scalebox{1.25}{
        $\InfminM_\sigma>0$
    }};
    \node (oneminm) at (8, -4) [draw=gray] {\scalebox{1.25}{
        $\InfoneminM_\sigma>0$
    }};
    \node (twooneminm) at (12, -6) [draw=gray] {\scalebox{1.25}{
        $2\oneminM_\sigma>0$
    }};
\end{scope}

\draw[-{Latex[length=3mm,width=4mm,open]}] 
    ([xshift=0mm,yshift=1mm]minm.north) to [bend right] 
    node[sloped,midway,above=2mm,align=center] {\cref{thm:infinite-minmax}\\ (via \cite{PT})} 
    ([xshift=1mm,yshift=0mm]besicovitch.east);
    
\draw[-{Latex[length=3mm,width=4mm,open]}] 
    ([xshift=0mm,yshift=1mm]oneminm.north) to [bend right]
    node[sloped,midway,above=2mm] {\cref{lem:easy-properties-minmax}} 
    ([xshift=1mm,yshift=0mm]minm.east);
    
\draw[-{Latex[length=3mm,width=4mm,open]}] 
    ([xshift=0mm,yshift=1mm]twooneminm.north) to [bend right] 
    node[sloped,midway,above=2mm] {\cref{lem:easy-properties-minmax}} 
    ([xshift=1mm,yshift=0mm]oneminm.east);

\node at ([xshift=0mm,yshift=-7mm]besicovitch) {False for $\sigma< 0.5$};

\node at ([xshift=-5mm,yshift=-7mm]oneminm) {False for $\sigma<0.64368\ldots$~\cref{thm:0.64}};

\node at ([xshift=0mm,yshift=-7mm]twooneminm) {False for $\sigma<0.683$~\cref{thm:0.7}};

\node at ([xshift=0mm,yshift=-12mm]twooneminm) {True for $\sigma\ge 0.7$~\cref{thm:0.7}};

\end{tikzpicture}

\caption{Schematic outline of the proof of $\bar\sigma\le 0.7$~\cref{t:best-bound}. 
The diagram illustrates the key steps of our proof, along with negative results that limit (certain parts of) our approach from proving $\bar\sigma = 0.5$~\cref{c:Bes}. It is plausible that utilizing $k\oneminM_\sigma$ for $k>2$ could significantly enhance our bound, potentially reaching $\bar\sigma\leq 0.64368\ldots$, though no further improvement is possible with this strategy. Conversely, it remains uncertain whether $\InfminM_\sigma$ can be used to establish the full conjecture. 
}
\end{figure}
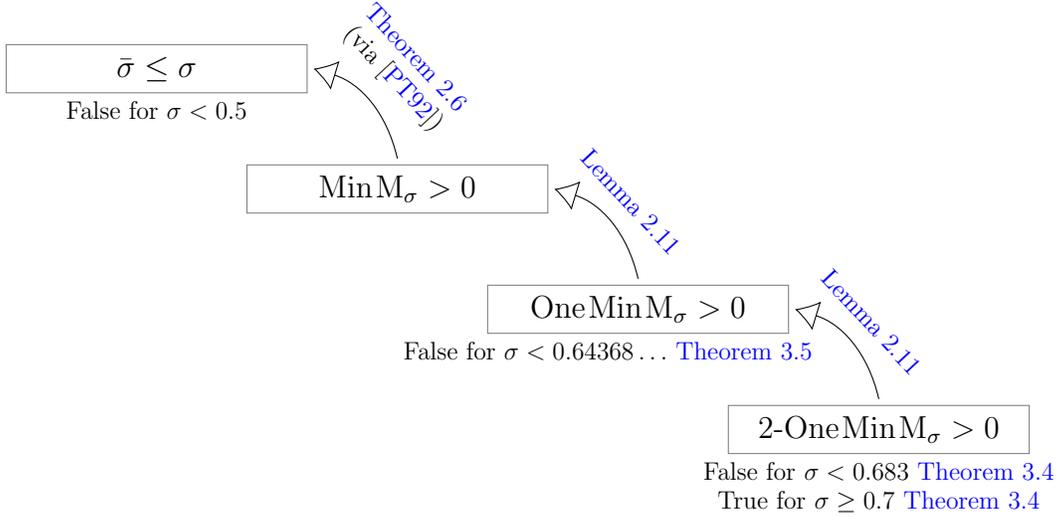

Besicovitch himself proved first in \cite{Besicovitch0} that $\bar\sigma \leq 1-10^{-2576}$ and later in \cite[Theorem 14]{Besicovitch} that $\bar\sigma \leq \frac{3}{4}$. The lower bound $\bar\sigma \geq \frac{1}{2}$ is given in \cite[Theorem 2, Section 17]{Dickinson} through a purely unrectifiable closed set $E$ with $\Theta^1_* (E, x) = \frac{1}{2}$ for $\mathcal{H}^1$-a.e. $x\in E$ (the set was suggested earlier by Besicovitch in \cite[Section 19]{Besicovitch}).
Since Besicovitch's pioneering work, his upper bound was improved in \cite{PT}:
$\bar \sigma \leq \frac{2+\sqrt{46}}{12} = 0.73186\ldots$
(in another direction, the sharp bound conjectured by Besicovitch has been proved to hold provided that the set $E$ satisfies some additional condition, see \cite{Farag-thesis,Farag1,Farag-2}). 

The main contribution of this paper is the introduction of a methodology to find upper bounds for $\bar\sigma$. In particular we define an explicit and simple infinite dimensional $\min$-$\max$ problem, which depends on a real parameter $\sigma$ and whose $\min$-$\max$ value $\InfminM_\sigma$ has the following key properties: $\sigma\mapsto \InfminM_\sigma$ is nondecreasing and nonnegative and, most importantly, $\InfminM_\sigma>0$ implies $\bar{\sigma} \leq \sigma$. For the relevant definition we refer the reader to \cref{s:statements}, where the proof of \cref{t:best-bound} is then reduced to showing $\InfminM_{0.7}>0$. 

Computing $\InfminM_\sigma$ is challenging and we thus resort to finding suitable lower bounds. This is achieved in several steps. The first is to find a second infinite-dimensional min-max problem, which we denote by $\InfoneminM_\sigma$, that satisfies $\InfminM_\sigma\geq \InfoneminM_\sigma$. However, while we do not know whether one could achieve Besicovitch's conjecture showing that $\InfminM_\sigma>0$ for every $\sigma>\frac{1}{2}$, we can exclude that $\InfoneminM$ is powerful enough, as we show that $\InfoneminM_{0.64368\ldots} = 0$. This limitation is the reason why we present both.

In the second step we define a family of finite-dimensional approximations of $\InfoneminM_\sigma$, parametrized by $k\in\N_0$ and denoted by $k\oneminM_\sigma$. These functions lie all below $\InfoneminM_\sigma$ and they approach it monotonically in $k$ as $k\uparrow \infty$. These problems consist all of finding the $\min$-$\max$ of a Lipschitz function on some suitable compact sets $K(k) \subset \mathbb R^{n(k)}$, where $n(k)$ grows with $k$. We then use them to show that
\begin{enumerate}\setcounter{enumi}{-1}
    \item $0\oneminM_\sigma>0$ if and only if $\sigma>\frac34$, hence recovering Besicovitch's result, cf. \cref{thm:0.75}.
    \item $1\oneminM_\sigma>0$ if and only if $\sigma> \sigma_{PT}=0.72655\ldots$, the unique positive zero of the third-order polynomial $8 s^3 + 4 s^2 - 3s-3$, cf. \cref{thm:0.726}. This case can be understood as an optimization of the strategy of Preiss and Ti\v{s}er.\footnote{After we completed this work, we were made aware by David Preiss that an optimization of their strategy was indeed performed by A. Schechter in his unpublished Diplomarbeit, \cite{Schechter}. Schechter does not formulate our optimization problem $1\oneminM$ but rather refines the case analysis and computations of \cite[pp. 285--287]{PT}. In particular, he finds $\sigma_{PT}$ solving a system of two algebraic equations. Unraveling Cardano's formula one can check that his specific expression (see e.g. \cite[p. 9]{Schechter}) does coincide with our definition of $\sigma_{PT}$. Therefore, his arguments can be translated into the ``if'' part of our statement (1) on $1\oneminM_\sigma$.}
    \item Finally, $2\oneminM_{0.7}>0$ (cf. \cref{thm:0.7}), which implies \cref{t:best-bound}. 
\end{enumerate}
In fact, ``Besicovitch's'' case becomes trivial in this setting, as $n(0)=1$ and the ``min'' part of the variational problem is absent. The ``Preiss-Ti\v{s}er'' case is the first nontrivial one: $n(1)=3+4=7$ (and the maximization is over three of the seven variables) but we can still compute explicitly the function $1\oneminM_\sigma$. Since $n(2) = 7 + 12 = 19$, computing $2\oneminM_\sigma$ seems an unreachable task for a human ``exact computation'': in order to give our rigorous estimate we need the assistance of a computer, which examines a (very large) finite number of cases. In this regard an important observation is that the ``max'' in all these min-maxes can be reduced to a collection of linear programming problems~\cite{Schrijver1986}, for which efficient algorithms are known. Even so, the dimensionality of the problem is so high that we need to solve nontrivial issues to make a computer-assisted proof feasible. 

Our paper contains further results. First of all we give lower bounds for what
\begin{itemize}
\item[(a)] $2\oneminM$ could achieve in \cref{c:Bes}, as we show $2\oneminM_{0.683}=0$, cf. \cref{thm:0.7};
\item[(b)] $\InfoneminM$ could achieve, showing $\InfoneminM_{\sigma}=0$ for $\sigma=0.64368\ldots$, the unique positive zero of the third-order polynomial $32s^3-32s^2+12s-3$, cf. \cref{thm:0.64}. We have some weak experimental evidence that this may be the sharp bound. In other words, there are reasons to believe that $\InfoneminM_\sigma>0$ for all $\sigma$ larger than this value. If this conjecture were true, it would imply $\bar\sigma\le 0.64368\ldots$.
\end{itemize}
Secondly, following \cite{PT} we introduce the Besicovitch number $\bar{\sigma} (X,d)$ of a metric space and we analyze the metric generalization of Besicovitch's conjecture. In this case we just introduce the metric analogs of $k\oneminM_\sigma$ to keep our presentation simpler. We then show that 
\begin{itemize}
\item[(c)] $1\oneminM_\sigma>0$ for every metric space and every $\sigma>\sigma_{PT}$, in particular $\bar{\sigma} (X,d) \leq \sigma_{PT}$ for every metric space $(X,d)$ (cf. \cref{thm:01-minmax-metric});
\item[(d)] There is a space $(X,d)$ for which $k\oneminM_{\sigma_{PT}} = 1\oneminM_{\sigma_{PT}} = 0$ for every integer $k\geq 1$, cf. \cref{thm:0.726-metric}.
\end{itemize}

Finally, we propose a ``quantitative version'' of Besicovitch's conjecture, cf. \cref{c:Bes-reg}. We cannot prove that the latter is implied or implies \cref{c:Bes}. However, all the results mentioned in this introduction translate equally well to \cref{c:Bes-reg}. The latter has the advantage that it can be formulated in a very elementary way in terms of connectedness properties of sets, without any reference to rectifiability. On the other hand we believe that it captures the essence of the problem, because $1$-dimensional rectifiable sets can be characterized as being ``big pieces'' of compact connected sets with finite length at most points and at sufficiently small scales. For the relevant discussion we defer to \cref{s:quantitative}.

\subsection*{Acknowledgments} The authors are thankful to Javier Gómez-Serrano for some discussions concerning the execution of scripts for computer-assisted proofs and to Peter M\"orters for sharing with them his personal copy of A. Schechter's Diplomarbeit \cite{Schechter}.

The third named author, A. M., has been supported by University of Padova's research programme STARS@unipd through project ``QuASAR – Questions About Structure And Regularity of currents''
(MASS STARS MUR22 01). The authors A. M. and D. V. have been partially supported by GNAMPA-INdAM and PRIN 2022PJ9EFL ``Geometric Measure Theory: Structure of Singular Measures, Regularity Theory and Applications in the Calculus of Variations''. Last but not least, this material is based upon work supported by the National Science Foundation
under Grant No. DMS-1926686.  


\section{Optimization problems}\label{s:statements}

In this section we describe the optimization problems whose solutions produce upper bounds for $\bar\sigma$. Fix $\sigma\in[\frac12,1]$.
We will first introduce the relevant objective function, called $F_\sigma$, which depends on a set of points $P\subseteq\R^2$, and a corresponding family of radii (one for each point). 
We will then maximize over a suitable space of these radii, gaining a corresponding function $M_\sigma$ which depends only on the set of points. Then, we will describe the family of set of points over which one shall compute the infimum of $M_\sigma$.
We will claim (and postpone the proof) that if such an infinite dimensional $\min$-$\max$ problem has a positive value then $\bar\sigma\le\sigma$.
The final part of this section is concerned with describing a relaxation of the above-mentioned infinite dimensional $\min$-$\max$ problem to a hierarchy of finite dimensional $\min$-$\max$ problems. These finite dimensional min-max problems are more treatable and are the ones we will use to show our main results.

\subsection{The objective function} \label{subsec:objective} Given a set of points $P\subseteq\R^2$, we define the following family of radii. 
Let
\[
\mathcal{R} (P):=\left\{ r:P\to[0,1]: \  
\begin{aligned}
    &r(p) > 0 \text{ for at most finitely many $p\in P$,}\\
    &B_{r(p)} (p)\cap B_{r(p')} (p') = \emptyset \;\; \forall p, p'\in P \text{ distinct}
\end{aligned}
\right\} \,.
\]
Here and in what follows $B_\tau (q)\subseteq\R^2$ is the open ball of radius $\tau$ centered at $q$. We understand $B_0(q)$ as the empty set; in particular, if $r(p)=0$, then another ball $B_{r(p')}(p')$ is allowed to contain $p$.

Next, for every $r\in \mathcal{R} (P)$ we set 
\begin{align*}
U(P, r) &:= \bigcup_{p\in P} B_{r(p)} (p) ,\\
R(P, r) &:= \inf \left\{R>0: B_R (O) \supset U(P,r) \right\} ,
\end{align*}
where $O=(0, 0)$ denotes the origin.
As above, since we understand $B_0(p)=\emptyset$, if $r(p)=0$ then the set $U(P, r)$ does not necessarily contain all the points $p\in P$ (cf. \cref{fig:U-and-R}).

\begin{figure}
\begin{tikzpicture}
\coordinate (P1) at (0,0);
\coordinate (P2) at (1.414,-1.414);
\coordinate (P3) at (-2.2,0);
\coordinate (P4) at (1.414, 1.414);
\coordinate (P5) at (0.9,2);
\coordinate (P6) at (2.7,2.7);

\filldraw[black] (P1) circle (1pt);
\filldraw[black] (P2) circle (1pt);
\filldraw[black] (P3) circle (1pt);
\filldraw[black] (P4) circle (1pt);
\filldraw[black] (P5) circle (1pt);
\filldraw[black] (P6) circle (1pt);

\node[below] at (P1) {$p_1=O$};
\node[left] at (P2) {$p_2$};
\node[below] at (P3) {$p_3$};
\node[below] at (P4) {$p_4$};
\node[below] at (P5) {$p_5$};
\node[below] at (P6) {$p_6$};

\draw[fill=gray, opacity=0.2] (P1) circle [radius=1];
\draw[fill=gray, opacity=0.2] (P3) circle [radius=0.7];
\draw[fill=gray, opacity=0.2] (P4) circle [radius=1];

\draw (P1) -- (1,0);
\draw (P3) -- (-1.5,0);
\draw (P4) -- (2.414, 1.414);

\draw (P1) [dashed] circle [radius = 3];
\draw (P1) [dashed] -- ({-3*1.414/2}, {3*1.414/2});

\node[above] at (-1.85,0) {$r_3$};
\node[above] at (0.5,0) {$r_1$};
\node[above] at (1.914,1.414) {$r_4$};
\node[right] at (-1.5,1.5) {$R(P, r)$};
\end{tikzpicture}
\caption{An example of $P$ consisting of six points and $U (P,r)$ consisting of three circles (the gray region). The dashed circle is the circle of radius $R(P,r)$ centered at $O$ (observe that $p_1$ coincides with $O$). For notational simplicity, we denote $r_i=r(p_i)$.
In this example $r_2=r_5=r_6=0$. While $U(P,r)$ contains the point $p_5$ (that is internal to $B_{r_4}(p_4)$), it does not contain $p_2$ nor $p_6$. It also happens that the ball $B_{R(P,r)}(p_0)$ does not contain all points of $P$; see the location of $p_6$.}
\label{fig:U-and-R}
\end{figure}
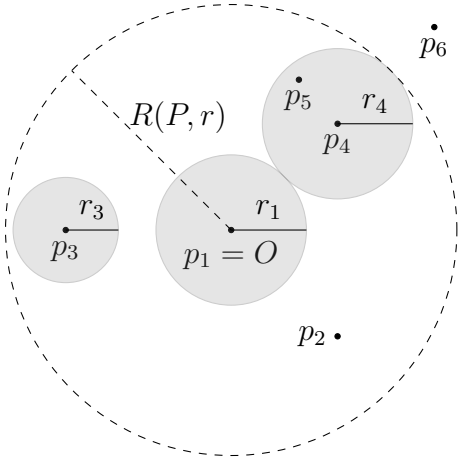

Observe that, thanks to the Euclidean structure of $\R^2$, we have the explicit formulas:
\begin{equation}\begin{aligned} \label{eq:RU-formulas}
R (P,r) &= \sup \{ |p| + r(p): \, p\in P \text{ with } r(p)>0\},\\
\diam\, (U (P,r)) &= \sup \{|p-p'|+r(p)+r(p') :\, p, p'\in P \text{ with } r(p), r(p') >0\}.
\end{aligned}\end{equation}

We are now ready to define the pivotal objective function, which depends on $\sigma \in [\frac{1}{2}, 1]$, on a {\em bounded} set $P\subseteq \R^2$, and a family of radii $r \in \mathcal{R} (P)$:
\[
F_\sigma (P, r)
:= \sum_{p\in P} r(p) - \tfrac1{2\sigma}\min \left\{ \diam\, (U (P, r)), \tfrac{1}{2} + R(P,r) \right\}\, .
\]
We agree that, if $r(p)=0$ for all $p\in P$, then $\diam\, (U (P, r))= R(P,r)=0$. Note that the set $P$ is not even assumed to be countable, however $r(p)$ can only be positive for a finite number of points, given our definition of $\mathcal{R} (P)$, hence the function is well defined.

The first part of the optimization problem consists of maximizing over $r\in \mathcal{R} (P)$:
\begin{equation}\label{e:def-M}
M_\sigma(P) :=
\max \{F_\sigma(P, r): r \in \mathcal{R} (P)\} \, .
\end{equation}

\begin{remark}\label{r:useful}
We immediately point out the useful property that $M_\sigma$ is monotone, i.e., if $P'\subseteq P$ then $M_\sigma(P')\le M_\sigma(P)$. Indeed, given $r'\in \mathcal R(P')$ let $r:P\to [0, 1]$ be the function such that $r=r'$ on $P'$ and $r=0$ on $P\setminus P'$. We have $r\in \mathcal R(P)$ and $F_\sigma(P', r') = F_\sigma(P, r)$ and taking the supremum over all choices of $r'\in\mathcal R(P')$ we get $M_\sigma(P')\le M_\sigma(P)$.
Furthermore, for any $P\subseteq\R^2$, we have $M_\sigma(P)\ge 0$ (choosing $r$ constantly equal to $0$, so that $U(P, r)=\emptyset$) and $M_\sigma(P)\le M_{\sigma'}(P)$ whenever $\sigma<\sigma'$.
\end{remark}

\begin{remark}
    For our purposes, it would have been equivalent, and slightly more intuitive, to consider 
    \begin{equation*}
        F(P, r) := \frac{\sum_{p\in P}r(p)}
        {\min\{
            \diam(U(P, r)), \tfrac12 + R(P, r)
        \}}
    \end{equation*}
    instead of $F_\sigma$ and then $M(P):= \sup_{r\in\mathcal R(P)} F(P, r)$ instead of $M_\sigma$.
    Indeed, $F_\sigma(P,r )\ge 0$ if and only if $F(P, r)\ge \frac1{2\sigma}$ and we will care only about the sign of $F_\sigma$ (see, for example, the statement of \cref{thm:infinite-minmax}).
    The objective function $M_\sigma$ is to be preferred because
    \begin{itemize}
        \item The continuity properties of $F_\sigma$ and $M_\sigma$ are easier to establish (and the proofs are less cumbersome as we avoid denominators).
        \item As we will see in \cref{sec:linear-programming}, $M_\sigma$ can be written as a collection of linear programming problems. This remains true even for $M$, but the corresponding problems are more complicated and the derivation is more involved.
    \end{itemize}
\end{remark}

\subsection{The search space: the stable sets of points}
Next, we shall define the family of set of points $P\subseteq\R^2$ over which we will compute the infimum of $M_\sigma$ to obtain the sought $\min$-$\max$ problem.

\begin{definition}
    Fix $\sigma\in[\frac12, 1]$.
    Given $p\in\R^2$ and $0<r\le 1$, let $\Delta_\sigma(p, r)$ be the family of pairs $(q_1, q_2)\in \overline{B_r(p)} \times \overline{B_r(p)}$ such that $|q_1-q_2|\ge 2\sigma r$.
\end{definition}

\begin{figure}
\begin{tikzpicture}
\filldraw[black] (0,0) circle (1pt);
\node[left] at (0,0) {$p$};
\draw (0,0) circle [radius = 3];

\draw (0,0) [dashed] -- ({3*1.414/2}, {-3*1.414/2});
\node[right] at (1.1,-1) {$r$};

\filldraw[black] (2.2,1.414) circle (1pt);
\filldraw[black] (-1.8,1.414) circle (1pt);

\node[below] at (-1.8,1.414) {$q_1$};
\node[below] at (2.2,1.414) {$q_2$};
\draw (2.2,1.414) -- (-1.8,1.414);
\node[above] at (0,1.414) {$\geq 2 \sigma r$};
\end{tikzpicture}
\caption{A pair $(q_1, q_2)$ belonging to $\Delta_\sigma(p, r)$. Observe that $q_1, q_2$ are inside $\overline{B_r(p)}$ and satisfy $|q_1-q_2|\geq 2\sigma r$.}\label{f:figli} 
\end{figure}
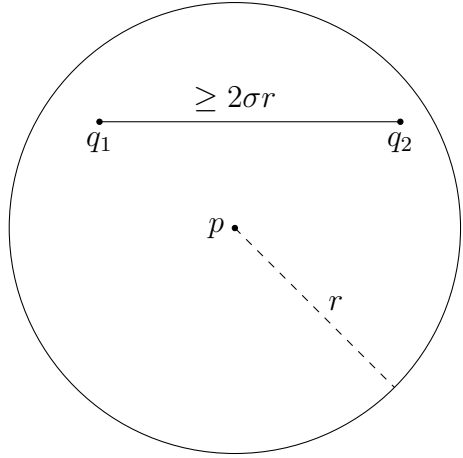

We are interested in sets that are \emph{closed} under the $\Delta_\sigma$ operations defined above.
\begin{definition}[stability]\label{def:stability}
    Fix $\sigma\in[\frac12, 1]$. For $L>1$, a set of points $P\subseteq\R^2$ is \emph{$(\sigma, L)$-stable} if the following statement holds:
\begin{itemize}
\item[(St)]    For any $p\in P \cap B_L(O)$ and any $r\in (L^{-1}, 1]$, there is a pair $(q_1, q_2)\in \Delta_\sigma(p, r)$ such that $q_1, q_2\in P$.
\end{itemize}
A set is called $(\sigma,\infty)$-stable if it is $(\sigma,L)$-stable for all $L>1$.
\end{definition}

A weaker notion which will play a particularly important role is the following.

\begin{definition}[one-stability]
    Fix $\sigma\in[\frac12, 1]$. For $L>1$, a set of points $P\subseteq\R^2$ is \emph{$(\sigma, L)$-one-stable} if the following statement holds:
\begin{itemize}
\item[(1-St)] For any $p\in P \cap B_L(O)$ there is a pair $(q_1, q_2)\in \Delta_\sigma(p, 1)$ such that $q_1, q_2\in P$.
\end{itemize}
As above a set is called $(\sigma,\infty)$-one-stable if it is $(\sigma,L)$-one-stable for all $L>1$.
\end{definition}

\subsection{The infinite dimensional min-max problem}
We are ready to state the infinite dimensional min-max problem that produces an upper bound for $\bar\sigma$. This is a central result in our work.

\begin{theorem}\label{thm:infinite-minmax}
    Fix $\sigma\in[\frac12,1]$. For $L>1$, define\footnote{Observe that $\InfminM_\sigma(L)$ is increasing in $L$ because being $(\sigma,L)$-stable is more restrictive when $L$ is larger. Therefore the limit $\InfminM_\sigma$ exists.}
    \begin{align*}
        \InfminM_{\sigma}(L) &:= \inf \{M_\sigma(P): P\text{ is $(\sigma, L)$-stable and $O\in P$}\}\, , \\
        \InfminM_\sigma &:= \lim_{L\to\infty}\InfminM_\sigma(L).    
    \end{align*}
    If $\InfminM_\sigma>0$, then $\bar\sigma \leq \sigma$.
\end{theorem}
We do not know if this statement is robust enough to prove \cref{c:Bes} (i.e., $\bar\sigma = \frac12$). More precisely, we do not know of a $\sigma>\frac12$ so that $\InfminM_{\sigma}=0$.

As an immediate corollary of \cref{thm:infinite-minmax} we have
\begin{corollary}\label{cor:infinite-minmax-one}
    Fix $\sigma\in[\frac12,1]$. For $L>1$, define\footnote{Observe that $\InfoneminM_\sigma(L)$ is increasing in $L$ because being $(\sigma,L)$-one-stable is more restrictive when $L$ is larger. Therefore the limit $\InfoneminM_\sigma$ exists.}
    \begin{align*}
        \InfoneminM_\sigma(L) & := \inf \{M_\sigma(P): P\text{ is $(\sigma, L)$-one-stable and $O\in P$}\}\,, \\
        \InfoneminM_\sigma &:= \lim_{L\to\infty} \InfoneminM_\sigma(L).
    \end{align*}
    If $ \InfoneminM_\sigma > 0$, then $\bar\sigma\le \sigma$.
\end{corollary}
\begin{proof}
    Observe that any $(\sigma, L)$-stable set is also $(\sigma, L)$-one-stable. Therefore in this statement we are computing the infimum over a larger family of sets of points than in \cref{thm:infinite-minmax}, which implies $\InfminM_\sigma \geq \InfoneminM_\sigma$. Hence the conclusion follows from \cref{thm:infinite-minmax}.
\end{proof}

\begin{remark}\label{r:mandorla} In fact the proof of \cref{thm:infinite-minmax} gives a slightly stronger statement. We can fix an arbitrary point $P_\star\in \partial B_1 (O)$ and define
\begin{align*}
\overline{\InfminM}_\sigma (L) &:= \inf \{M_\sigma(P): O\in P\text{ is $(\sigma, L)$-stable and } B_1 (P_\star)\cap P=\emptyset\}\\
\overline{\InfminM}_\sigma &:= \lim_{L\to \infty} \overline{\InfminM}_\sigma (L)\, .
\end{align*}
Obviously $\overline{\InfminM}_\sigma \geq \InfminM_\sigma$, nonetheless we still have that $\overline{\InfminM}_\sigma>0$ implies $\bar\sigma \leq \sigma$. Likewise we can introduce $\overline{\InfoneminM}_\sigma$, the corresponding counterpart of $\InfoneminM_\sigma$.

It does not seem that this additional constraint can be used effectively, while it would make several aspects much more technical. For the latter reason we will mostly ignore it (see \cref{r:mandorla2} for a detailed explanation): we will only keep track of it in our examples, because we can show that they comply as well with this additional restriction, cf. \cref{r:mandorla3,r:mandorla4}.
\end{remark}

\subsection{A hierarchy of finite dimensional min-max problems}
The goal of this section is to define a sequence of $\min$-$\max$ problems providing a discrete approximation of $\InfoneminM_\sigma$.

Let us start by defining a sequence of families $(\mathcal F_\sigma(k))_{k\in \N_0}$ of finite sets of points that are a discrete analogue of one-stable sets.

Fix $\sigma\in[\frac12,1]$. 
Define
\begin{equation}\label{eq:F-sigma-k}\begin{aligned}
    \mathcal F_\sigma(0) &:= \{\{O\}\}, \\
    \mathcal F_\sigma(k+1) &:= 
    \Big\{
        P \cup \bigcup_{p\in P} \{q_{p,1}, q_{p,2}\}:\,
        P\in \mathcal F_\sigma(k),\, (q_{p,1}, q_{p,2})\in \Delta_\sigma(p, 1)\text{ for all $p\in P$}
    \Big\}.
\end{aligned}\end{equation}

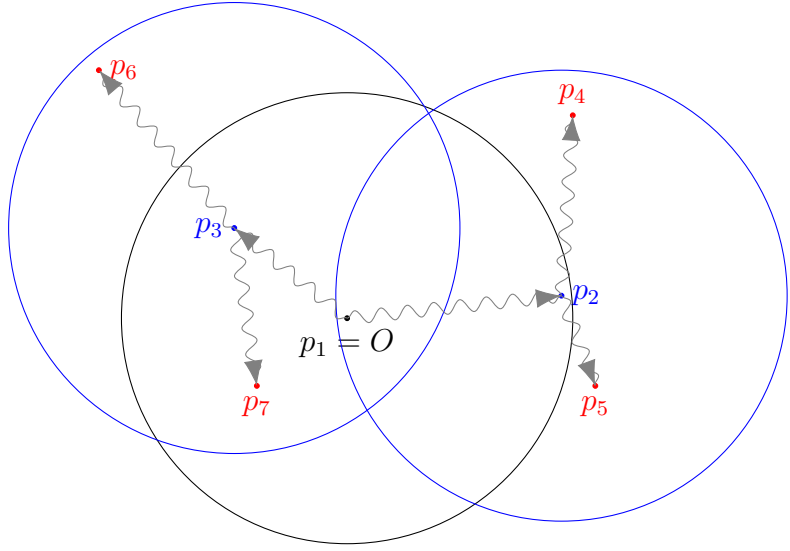
\begin{figure}
\begin{tikzpicture}[scale=3,decoration=snake]
\coordinate (P1) at (0,0);
\coordinate (P2) at (0.95,0.1);
\coordinate (P3) at (-0.5,0.4);
\coordinate (P4) at (1, 0.9);
\coordinate (P5) at (1.1,-0.3);
\coordinate (P6) at (-1.1,1.1);
\coordinate (P7) at (-0.4,-0.3);

\filldraw[black] (P1) circle (0.3pt);
\filldraw[blue] (P2) circle (0.3pt);
\filldraw[blue] (P3) circle (0.3pt);
\filldraw[red] (P4) circle (0.3pt);
\filldraw[red] (P5) circle (0.3pt);
\filldraw[red] (P6) circle (0.3pt);
\filldraw[red] (P7) circle (0.3pt);

\node[below] at (P1) {$p_1=O$};
\node[right,color=blue] at (P2) {$p_2$};
\node[left,color=blue] at (P3) {$p_3$};
\node[above,color=red] at (P4) {$p_4$};
\node[below,color=red] at (P5) {$p_5$};
\node[right,color=red] at (P6) {$p_6$};
\node[below,color=red] at (P7) {$p_7$};

\draw (P1) circle [radius=1];
\draw[color=blue] (P2) circle [radius=1];
\draw[color=blue] (P3) circle [radius=1];

\draw[-{Latex[scale=2]},decorate,color=gray] (P1) -- (P2);
\draw[-{Latex[scale=2]},decorate,color=gray] (P1) -- (P3);
\draw[-{Latex[scale=2]},decorate,color=gray] (P2) -- (P4);
\draw[-{Latex[scale=2]},decorate,color=gray] (P2) -- (P5);
\draw[-{Latex[scale=2]},decorate,color=gray] (P3) -- (P6);
\draw[-{Latex[scale=2]},decorate,color=gray] (P3) -- (P7);
\end{tikzpicture}
\caption{An example of a $7$-points set $P$ belonging to $\mathcal F_\sigma(2)$. 
The set $\{p_1=O\}$ belongs to $\mathcal F_\sigma(0)$. The set $\{p_1, p_2, p_3\}$ belongs to $\mathcal F_\sigma(1)$. The set $\{p_1, p_2, p_3, p_4, p_5, p_6, p_7\}$ belongs to $\mathcal F_\sigma(2)$. In particular, $(p_2, p_3)\in\Delta_\sigma(p_1, 1)$, $(p_4,p_5)\in\Delta_\sigma(p_2, 1)$, and $(p_6, p_7)\in\Delta_\sigma(p_3, 1)$.
The three circles are the unit circles centered at $p_1, p_2, p_3$. The points satisfy $|p_2-p_3|\ge 2\sigma$, $|p_4-p_5|\ge 2\sigma$, and $|p_6-p_7|\ge 2\sigma$.
}
\label{fig:generations}
\end{figure}

One-stable sets are related to $(\mathcal F_\sigma(k))_{k\in \N_0}$ thanks to the following observation.
\begin{lemma}\label{lem:stable-has-finite-subset}
    Fix $\sigma\in[\frac12,1]$.
    Given $L>0$, let $P\subseteq\R^2$ be a $(\sigma, L)$-one-stable set that contains the origin $O$. 
    For any integer $0\le k\le L$, there is $P_k\subseteq P$ such that $P_k\in\mathcal F_\sigma(k)$.
\end{lemma}
\begin{proof}
    We prove the statement by induction over $k\leq L$. 
    For $k=0$, we choose $P_0=\{O\}$.

    Now, assume that some fixed integer $k\leq L$ the statement of the lemma holds. If $k> L-1$ we are finished. If $k\leq L-1$ we want to show that the statement holds for $k+1$ as well. Hence fix a $P_k$ which fulfills the requirements of the lemma.
    Observe that, by definition of $\mathcal F_\sigma(k)$, we have $P_k\subseteq B_{k}(O)\subseteq B_{L-1}(O)$.
    In order to produce $P_{k+1}$ proceed as follows.
    For each $p\in P_k$, since $p\in P\cap B_L(O)$, there exists a pair of points $(q_{p,1}, q_{p,2})\in\Delta_\sigma(p, 1)$ such that $\{q_{p,1}, q_{p,2}\}\subseteq P$. Define
    \begin{equation*}
        P_{k+1} := P_k \cup \bigcup_{p\in P_k} \{q_{p,1}, q_{p,2}\}.
    \end{equation*}
    By definition of $\mathcal F_\sigma(k+1)$, we have $P_{k+1}\in\mathcal F_\sigma(k+1)$ as desired.
\end{proof}

As a consequence of the latter lemma, we deduce the following relaxation of \cref{cor:infinite-minmax-one}.
\begin{corollary}\label{cor:finite-minmax}
    Fix $\sigma\in[\frac12,1]$. 
    For any $k\in \N_0$, if
    \begin{equation*}
        k\oneminM_\sigma := \inf_{P\in\mathcal F_\sigma(k)} M_\sigma(P) > 0
    \end{equation*}
    then $\bar\sigma \leq \sigma$.
\end{corollary}
\begin{proof}
    For any $(\sigma, k)$-one-stable $P\subseteq\R^2$ with $O\in P$, thanks to \cref{lem:stable-has-finite-subset} we can find $P_k\subseteq P$ such that $P_k\in \mathcal F_\sigma(k)$. In particular, since $M_\sigma$ is increasing over its argument, we have $M_\sigma(P_k)\le M_\sigma(P)$. Therefore, we deduce
    \begin{equation*}
        \inf_{P\in\mathcal F_\sigma(k)} M_\sigma(P) 
        \le 
        \inf_{O\in P\text{ is $(\sigma, k)$-one-stable}}
            M_\sigma(P) = \InfoneminM_\sigma(k) \leq \InfoneminM_\sigma\, .
    \end{equation*}
    Hence $\bar\sigma\leq \sigma$ follows from \cref{cor:infinite-minmax-one}.
\end{proof}

Let us conclude with a statement collecting many simple properties of the various $\min$-$\max$ problems.
\begin{lemma}\label{lem:easy-properties-minmax}
    The following inequalities hold.
    \begin{itemize}
        \item For $\tfrac 12\le \sigma \le 1$ and $1<L$,
        \begin{equation*}
            \InfminM_{\sigma}(L) \ge 
            \InfoneminM_{\sigma}(L),
        \end{equation*}
        and both values are nondecreasing in $\sigma$ and $L$.
        \item For $\tfrac 12\le\sigma\le 1$,
        \begin{equation*}
            \InfminM_{\sigma} \ge 
            \InfoneminM_{\sigma},
        \end{equation*}
        and both values are nondecreasing in $\sigma$.
        \item For $\tfrac 12\le \sigma\le 1$ and $k\in\N_0$,
        \begin{equation*}
            k\oneminM_\sigma 
            \le 
            \InfoneminM_{\sigma}(k)\le \InfoneminM_\sigma,
        \end{equation*}
        and the left-hand side is nondecreasing in $\sigma$ and $k$.
    \end{itemize}
\end{lemma}
\begin{proof}
    The first inequality is shown in the proof of \cref{cor:infinite-minmax-one}. The second inequality follows from the first one by passing to the limit as $L\to\infty$. The third inequality is shown in the proof of \cref{cor:finite-minmax}.
    The claimed monotonicities are direct consequences of the definitions of the various $\min$-$\max$ problems.
\end{proof}

\section{Main results about the min-max problems in \texorpdfstring{$\mathbb R^2$}{R2}}

As $k$ goes to infinity the finite dimensional problems $k\oneminM_\sigma$ recover $\InfoneminM_\sigma$. 

\begin{theorem}\label{thm:convergence-finite-minmax}
    For any $\sigma\in[\tfrac12, 1]$ 
    \begin{equation*}
        \lim_{\sigma'\uparrow\sigma} \InfoneminM_{\sigma'} 
        \le 
        \lim_{k\to \infty}k\oneminM_\sigma 
        \le 
        \InfoneminM_\sigma.
    \end{equation*}
\end{theorem}

We show the following bounds about the first three generations of the finite dimensional $\min$-$\max$ problems. 

\begin{theorem}[Besicovitch]\label{thm:0.75}
    Set $\sigma_B := \frac34 = 0.75$.
    For $\sigma\in[\tfrac12,1]$,
    \begin{equation*}
        0\oneminM_\sigma > 0 \quad\text{if and only if}\quad \sigma>\sigma_B.
    \end{equation*}
    Therefore (recall \cref{cor:finite-minmax}) $\bar\sigma \le\sigma_B$.
\end{theorem}

\begin{theorem}[Preiss--Ti\v{s}er on steroids]\label{thm:0.726}
    Let $\sigma_{PT}:= \frac{-2+\sqrt[3]{2 (131-9\sqrt{179})}+\sqrt[3]{2(131+9\sqrt{179})}}{12}\approx 0.72655\ldots$ be the unique positive solution of $8s^3 + 4s^2 - 3s-3=0$.
    For $\sigma\in[\tfrac12,1]$,
    \begin{equation*}
        1\oneminM_\sigma > 0 \quad\text{if and only if}\quad \sigma>\sigma_{PT}.
    \end{equation*}
    Therefore (recall \cref{cor:finite-minmax}) $\bar\sigma \le \sigma_{PT}$.
\end{theorem}

\begin{theorem}[Bounds on the $2$-nd generation]\label{thm:0.7}
    For $\sigma=0.7$,
    \begin{equation*}
        2\oneminM_\sigma > 0,
    \end{equation*}
    and therefore (recall \cref{cor:finite-minmax}) $\bar\sigma \le 0.7$,
    while for $\sigma=0.683$,
    \begin{equation*}
        2\oneminM_\sigma = 0.
    \end{equation*}
\end{theorem}

Notice that \cref{thm:0.7} implies immediately the main result of this note, that is \cref{t:best-bound}. We provide also a lower bound for those $\sigma$'s such that that $\InfoneminM_\sigma>0$, which in turn gives an analogous bound for the finite dimensional problems $k\oneminM_\sigma$, thanks to \cref{lem:easy-properties-minmax}. This is an obstruction to using $\InfoneminM_\sigma$ to prove \cref{c:Bes} but does not say anything about the feasibility of using $\InfminM_\sigma$ to that purpose.

\begin{theorem}[Lower bound on $\InfoneminM_\sigma$]\label{thm:0.64}
    Let $\sigma = 0.64368\dots$ be the unique positive solution of $32s^3-32s^2+12s-3=0$. We have,
    \begin{equation*}
        \InfoneminM_\sigma = 0.
    \end{equation*}
\end{theorem}

The proof of \cref{thm:convergence-finite-minmax} is contained in \cref{sec:cheap_continuity_arguments}.
The proof of \cref{thm:0.75}, which we include at the end of this section, is immediate (because $\mathcal F_\sigma(0)=\{\{O\}\}$).
The proof of \cref{thm:0.726} is an elementary---albeit involved---reduction to many simple inequalities (see~\cref{sec:0.726}). The arguments given in \cite{PT} can be used to prove \cref{thm:0.726}, i.e., that $1\oneminM_\sigma>0$ for $\sigma>\sigma_{PT}$. 
However the authors did not adopt our point of view and hence were not lead naturally, as we have been, to the sharp bound $\sigma_{PT}$ (their argument, recast in our framework, proves $1\oneminM_{0.73186\ldots}>0$).
The proof of the first statement in \cref{thm:0.7} is much more complicated. 
We will need a sequence of further observations to reduce the dimension of the search space, followed by a computationally-heavy computer-assisted proof by case exhaustion that consists of a fine subdivision of the search space together with a lower bound for $M_\sigma$ over each subdomain. The details are given in \cref{s:0.7}.

The \emph{negative results} (i.e., $2\oneminM_{0.683}=0$ in \cref{thm:0.7} and the statement of \cref{thm:0.64}) are proven in \cref{sec:explicit-counterexamples} by producing explicit small configurations of points.

\begin{proof}[Proof of \cref{thm:0.75}]
    We have $\mathcal F_\sigma(0)=\{\{O\}\}$.
    For $P=\{O\}$, $\mathcal R(P)$ coincides with the interval $[0, 1]$ (by identifying $r$ with $r(O)$). By unrolling the definitions, we have (for $\tfrac12<\sigma\le 1$)
    \begin{align*}
        0\oneminM_\sigma
        &=
        M_\sigma(\{O\})
        =
        \max_{0\le r\le 1} \left(r - \tfrac1{2\sigma}\min\big\{2r,\, \tfrac12 + r\big\}\right)
        \\
        &=
        \max_{0\le r\le 1}
        \max\Big\{
        r\big(1-\tfrac1{\sigma}\big),\,
        r\big(1-\tfrac1{2\sigma})-\tfrac1{4\sigma}
        \Big\}
        =
        \max\big\{0, 1-\tfrac{3}{4\sigma}\big\}.
    \end{align*}
    This exact formula for $0\oneminM_\sigma$ implies the desired statement.
\end{proof}

\section{Generalization to metric spaces}
Even though this note focuses on the case of $\R^2$, many of our results can be generalized to metric spaces. In this section we give a quick overview of which statements can be adapted to the metric setting. For an arbitrary metric space $(X, d)$, we let $\bar\sigma(X, d)$ be the constant analogous to $\bar\sigma$ if $\R^2$ is replaced by $(X, d)$ in \cref{def:bar-sigma}.

If one considers $\R^n$ instead of $\R^2$ as ambient space, all of our results continue to hold (with the exact same proofs) except for the first statement in \cref{thm:0.7}.
The case of an arbitrary metric space $(X, d)$ instead of $\R^2$ is more delicate because of the following issues:
\begin{enumerate}
    \item In our definitions and statements we used extensively a special point of $\R^2$ (i.e., the origin $O$) and a special value for certain distances (i.e., the value $1$), a strategy which works well thanks to the homogeneity properties of Euclidean spaces. 
    \item In the proof of \cref{thm:infinite-minmax}, which is the pillar on which everything else is built, we will leverage the local compactness and doubling properties of the Euclidean spaces (e.g., to prove \cref{lem:continuity-Delta,lem:continuity-sigma-stable,lem:sigma-stable-is-finite}).
\end{enumerate}
To solve the first issue, we shall add more parameters, at the expense of clarity (the origin $O$ becomes an arbitrary point, as well as the value $1$ becomes an arbitrary $D>0$).

To solve the second issue we shall skip the infinite dimensional problems $\InfminM_\sigma$ and $\InfoneminM_\sigma$ and state only the results for the finite dimensional problems. 
\begin{definition}
    Let $(X, d, O\in X)$ be a pointed metric space.
    
    For a set $P\subseteq X$, the definition~\cref{e:def-M} of $M_\sigma(P)$ makes perfect sense also in this more general setting (since we have chosen a special point $O\in X$).
    One may also define the families $(\mathcal F_\sigma^{(X, d, O)}(k))_{k\in\N_0}$ by copying the definition~\cref{eq:F-sigma-k} of the families $\mathcal F_\sigma(k)$.
    Therefore, for any $k\in\N_0$, one can define
    \begin{equation*}
        k\oneminM_\sigma(X, d, O) := \inf_{P\in \mathcal F_\sigma^{(X, d, O)}(k)} M_\sigma(P).
    \end{equation*}
\end{definition}

For an arbitrary metric space, one can use the proof of \cref{thm:infinite-minmax} to show directly \cref{cor:finite-minmax}. Going through this shortcut, all the technical difficulties that required local compactness and the doubling property disappear. Hence, we get the following statement, whose proof we skip because it is a strict subset of the proof of \cref{thm:infinite-minmax}.

\begin{theorem}\label{thm:finite-minmax-metric}
    Let $(X, d)$ be a complete metric space and fix $\sigma\in[\tfrac 12, 1]$.
    If we have
    \begin{equation*}
        \inf_{\substack{O\in X\\ D>0}} k\oneminM_\sigma(X, \tfrac{d}{D}, O) > 0
    \end{equation*}
    for some $k\in\N_0$, 
    then $\bar\sigma(X, d)\le \sigma$.
\end{theorem}

At this point, one can check that the proofs of \cref{thm:0.75} and \cref{thm:0.726} work verbatim for any metric space, leading to the following statement.
\begin{theorem}\label{thm:01-minmax-metric}
    Let $(X, d, O\in X)$ be a pointed metric space.
    For $\sigma\in[\tfrac 12, 1]$, we have
    \begin{align*}
        0\oneminM_\sigma(X, d, O) > 0 \quad\text{if $\sigma>\sigma_{B}$,}\\
        1\oneminM_\sigma(X, d, O) > 0 \quad\text{if $\sigma>\sigma_{PT}$,}
    \end{align*}
    where the constants $\sigma_B, \sigma_{PT}$ are defined respectively in \cref{thm:0.75} and \cref{thm:0.726}. In view of \cref{thm:finite-minmax-metric}, if $(X,d)$ is an arbitrary metric space, $\bar\sigma(X, d)\leq \sigma_{PT}\approx 0.72655\ldots$.
\end{theorem}

\begin{remark}
Note that, strictly speaking, in order to apply \cref{thm:finite-minmax-metric} we need the completeness of the metric space $(X,d)$. However, if $(X,d)$ is not complete, we can simply observe that $\bar\sigma (X,d) \leq \bar{\sigma} (X^c,d^c)$. where $(X^c,d^c)$ denotes its completion.
\end{remark}

\cref{thm:01-minmax-metric} leaves the door open to using the problems $k\oneminM_\sigma$ for $k>1$ to get a sharper bound on $\bar\sigma(X, d)$ for any metric space. This possibility is ruled out by the next statement which shows that --- at least for one particular metric space --- nothing is gained by considering $k\oneminM_\sigma$ for $k>1$ instead of the simpler $1\oneminM_\sigma$.

\begin{theorem}\label{thm:0.726-metric}
    There is a pointed metric space $(X, d, O)$ and a $4$-point set $O\in P\subseteq X$ which is $(\sigma_{PT},\infty)$-one-stable and such that $M_{\sigma_{PT}}(P) = 0$ (see \cref{thm:0.726} for the definition of $\sigma_{PT}$).  In particular, $k\oneminM_{\sigma_{PT}}(X, d, O)=0$ for all $k\in\N_0$.
\end{theorem}
The construction we use to show \cref{thm:0.726-metric}, which is deeply inspired by our proof of \cref{thm:0.726}, is contained in \cref{sec:explicit-counterexamples}.

\section{Continuity properties of the objects involved in the min-max}\label{sec:cheap_continuity_arguments}

In this section we study continuity and stability properties of the families of sets and the objective functions of our variational problems. In the first part we collect four lemmas about the behavior of $\Delta_\sigma (p,r)$ and $(\sigma, L)$-stable sets and prove \cref{thm:convergence-finite-minmax}. In the second part we prove a sharp Lipschitz bound for $M_\sigma$ which will play a crucial role in \cref{s:0.7}.

\subsection{Estimates on stable sets and perturbations}
Let us start with the following technical lemma about the \emph{continuity} of $\Delta_\sigma$-sets.
\begin{lemma}\label{lem:continuity-Delta}
    For any $\tfrac12\le \sigma' < \sigma\le 1$ and $0<r<r'<1$ such that $\sigma' r' < \sigma r$, there is $\eps=\eps_{\ref*{lem:continuity-Delta}}(\sigma,\sigma',r, r')>0$, depending continuously on its parameters, so that 
    \begin{equation*}
        \bigcup_{(q_1, q_2)\in\Delta_\sigma(p, r)}
        B_\eps(q_1)\times B_\eps(q_2)
        \subseteq \Delta_{\sigma'}(p', r') \qquad \forall p, p'\in\R^2 \text{ with } |p-p'|\le\eps\, .
    \end{equation*}
\end{lemma}
\begin{proof}
    Observe that $\Delta_{\sigma}(O, r)$ is compactly embedded in $\Delta_{\sigma'}(O, r')$ (interpreting them as subsets of $\R^4$).
    Let $\eps$ be one tenth of the distance between $\Delta_{\sigma}(O, r)$ and the complement of $\Delta_{\sigma'}(O, r')$ (observe that $\eps$ varies continuously with respect to $\sigma,\sigma',r, r'$).
    Because of our choice of $\eps$, we have
    \begin{equation*}
        \{(q_1+u, q_2+v)+w: (q_1, q_2)\in \Delta_\sigma(O, r),\, u, v, w\in \overline{B_{\eps}(O)}\} \subseteq \Delta_{\sigma'}(O, r').
    \end{equation*}
    Since $\Delta_\sigma(p, r)=(p, p) + \Delta_\sigma(O, r)$, it is not hard to verify that this choice of $\eps$ works.
\end{proof}

As a simple consequence we deduce that a stable set remains stable after a perturbation.
\begin{lemma}\label{lem:continuity-sigma-stable}
    For any $\frac12\le \sigma' < \sigma\le 1$ and $0 < L' < L$, there is $\eps=\eps_{\ref*{lem:continuity-sigma-stable}}(\sigma,\sigma',L,L')>0$ such that the following statement holds.

    Let $P\subseteq\R^2$ be a $(\sigma,L)$-stable set and let $P'\subseteq \R^2$ be another subset such that $d_H(P, P')< \eps$, where $d_H$ denotes the Hausdorff distance. 
    Then, $P'$ is a $(\sigma', L')$-stable set.
\end{lemma}
\begin{proof} For the moment fix $\eps$, whose choice will be specified along the argument.
Take $p'\in P'\cap B_{L'}(O)$ and $\tfrac1{L'}<r\le 1$. There is $p\in P'\cap B_{L'+\eps}(O)$ so that $|p-p'|\le\eps$.
    Fix $\delta>0$ so that $L'+\delta \le L$, $\tfrac 1{L'}-\delta >\tfrac1L$, and $\sigma'\tfrac1{L'} < \sigma (\tfrac1{L'}-\delta)$ (in particular $\sigma'r < \sigma (r-\delta)$). 
    We assume that $\eps<\delta$.
    Since $P$ is $(\sigma, L)$-stable, we can find $(q_1, q_2)\in \Delta_\sigma(p, r-\eps)$  such that $q_1, q_2\in P$. There are $q_1', q_2'\in P'$ so that $|q_1-q_1'|<\eps$ and $|q_2-q_2'|<\eps$.
    
    Assume in addition to $\varepsilon<\delta$ that 
    \begin{equation}\label{e:specify-eps-1}
        \eps < \min_{r\in \big[\tfrac 1{L'}, 1\big]} \eps_{\ref*{lem:continuity-Delta}}(\sigma,\sigma',r-\delta,r)\, .
    \end{equation} 
    Thanks to \cref{lem:continuity-Delta}, we deduce $(q_1', q_2')\in \Delta_{\sigma'}(p', r)$. Since $p'$ and $r$ were arbitrary, we have proven that $P'$ is $(\sigma', L')$-stable.
\end{proof}

In the following statement we prove that $(\sigma, L)$-stability is in fact a ``discrete notion''.
\begin{lemma}\label{lem:sigma-stable-is-finite}
    For any $\frac12 \le \sigma' \le \sigma\le 1$ and $0<L'<L$, there is a positive integer $N=N_{\ref*{lem:sigma-stable-is-finite}}(\sigma,\sigma',L,L')$ such that the following statement holds.
    
    Any $(\sigma,L)$-stable set $P\subseteq\R^2$ containing the origin admits a $(\sigma',L')$-stable subset $O\in P'\subseteq P\cap B_{L+1}(O)$ with at most $N$ points.
\end{lemma}
\begin{proof}
    Fix $\eps < \eps_{\ref*{lem:continuity-sigma-stable}}(\sigma,\sigma',L,L')$.
    Let $\tilde P\subseteq P$ be a subset such that:
    \begin{enumerate}
        \item Any two distinct points $x, y\in \tilde P$ satisfy $|x-y|>\eps$,
        \item For any $p\in P$ there is $x\in \tilde P$ so that $|x-p|\le \eps$.
    \end{enumerate}
    One way to construct $\tilde{P}$ is to start with the singleton $\bar P=\{O\}$ and inductively add to $\bar P$ the point $p\in P\setminus \bar P$ which minimizes $|p|$ and so that $\{p\} \cup \bar P$ does not violate the first condition. This procedure must stop after a finite number of additions and $\tilde{P}$ is the resulting set after the last addition.
    A consequence of the second condition is that $d_H(P, \tilde P)<\eps$ and thus, by applying \cref{lem:continuity-sigma-stable}, we obtain that $\tilde P$ is $(\sigma', L')$-stable.

    Let $P' := \tilde P\cap B_{L+1}(O)$. Since the notion of $(\sigma', L')$-stability does not take into account points outside of $B_{L'+1}(O)$, $P'$ mst be $(\sigma', L')$-stable. Moreover, the first condition satisfied by $\tilde P$ guarantees that $P'$ can have at most $C \frac{L^2}{\eps^2}$ elements, for a universal constant $C>0$.
\end{proof}

The following statement is a partial converse to \cref{lem:stable-has-finite-subset}.
\begin{lemma}\label{lem:finite-has-stable-subset}
    For any $\tfrac12\le\sigma'<\sigma\le 1$, any $L>0$, and any $1<r<2$ so that $\sigma'r<\sigma$, there is $\kappa=\kappa_{\ref*{lem:finite-has-stable-subset}}(\sigma,\sigma',L,r)\in\N_0$ such that that the following holds.

    Any $P\in\mathcal F_\sigma(k)$, with $k\ge\kappa$, admits a subset $O\in P'\subseteq P$ such that $\frac1r P'$ is $(\sigma', L)$-one-stable.
\end{lemma}
\begin{proof}
    We prove the statement for $\kappa$ equal to the maximum number of points that can be put in $B_{2L}(O)$ so that the pairwise distances are all greater than $\eps_{\ref*{lem:continuity-Delta}}(\sigma', \sigma, 1, r)$ (see \cref{lem:continuity-Delta}).

    Take $P\in \mathcal F_\sigma(k)$.
    By definition of $\mathcal F_\sigma(k)$, we can find $\{O\}=P_0\subseteq P_1 \subseteq\dots \subseteq P_k=P$ so that $P_i\in \mathcal F_\sigma(i)$ and for any $p\in P_i$ there are $(q_1, q_2)\in \Delta_\sigma(p, 1)$ such that $q_1, q_2\in P_{i+1}$.
    
    We construct a sequence of increasing subsets $O\in Q_i\subseteq P_i$.
    Set $Q_0:= \{O\}$. Then, for $0\le i<k$, define $Q_{i+1}$ as follows. 
    Choose arbitrarily $p_i\in Q_i\cap B_{2L}(O)$ so that $\Delta_{\sigma'}(p_i, r)\cap (Q_i\times Q_i) = \emptyset$. If such a point $p_i$ does not exist, end the construction.
    Otherwise set $Q_{i+1}=Q_i\cup\{q_1, q_2\}$, where $(q_1, q_2)\in\Delta_\sigma(p_i, 1)$ and $q_1, q_2\in P_{i+1}$.

    Let us show that the construction ends strictly before reaching the index $i=\kappa$. If $i<j$, by construction, we know that there are $(q_1, q_2)\in \Delta_\sigma(p_i, 1)$ with $q_1, q_2\in P_{i+1}\subseteq P_j$. Therefore, by definition of $p_j$, it must hold $(q_1, q_2)\not\in \Delta_{\sigma'}(p_j, r)$. In particular, $\Delta_\sigma(p_i, 1)\not\subseteq \Delta_{\sigma'}(p_j, r)$.
    Thus \cref{lem:continuity-Delta} implies that $|p_i-p_j|>\eps_{\ref*{lem:continuity-Delta}}(\sigma',\sigma,1,r)$. By definition of $\kappa$, we deduce that the construction ends before reaching the index $i=\kappa$.

    Define $P':= Q_i$ where $i$ is the index such that the construction ends at the $i$-th step. In particular, for any $p\in P'\cap B_{2L}(O)$, there are $(q_1, q_2)\in\Delta_{\sigma'}(p, r)$ so that $q_1, q_2\in P'$. This condition is equivalent the $(\sigma', \tfrac{2L}r)$-one-stability of $\frac1r P'$ which is strictly stronger than the desired $(\sigma', L)$-one-stability.
\end{proof}

This last lemma allows us to provide a short proof for \cref{thm:convergence-finite-minmax}.
\begin{proof}[Proof of \cref{thm:convergence-finite-minmax}]
    The inequality $\lim_{k\to\infty}k\oneminM_\sigma\le \InfoneminM_\sigma$ is a direct consequence of \cref{lem:easy-properties-minmax}. 

    To show the other inequality in the statement, choose $\tfrac12\le \sigma'<\sigma$ and $L>0$. Fix arbitrarily $1<r<\tfrac\sigma{\sigma'}$ and let  $\kappa:=\kappa_{\ref*{lem:finite-has-stable-subset}}(\sigma',\sigma, L, r)$. 

    Take $P\in\mathcal F_\sigma(k)$ for $k\ge\kappa$.
    Thanks to \cref{lem:finite-has-stable-subset}, we can find $O\in P'\subseteq P$ so that $\frac1r P'$ is $(\sigma', L)$-one-stable. Clearly we have
    \begin{equation}\label{eq:local53}
        M_{\sigma'}(\tfrac1r P') \le M_{\sigma}(\tfrac1r P').
    \end{equation}
If $M_\sigma (\tfrac{1}{r} (P'))=0$ then $M_\sigma (\tfrac{1}{r} P')\leq M_\sigma (P')$.
Otherwise, fix $\rho \in \mathcal{R} (\tfrac{1}{r} P')$ and consider $\bar\rho: P' \to [0, \infty)$ given by $\bar\rho (p) = r \rho (\frac{p}{r})$. We have 
\begin{equation}\label{eq:local54}
r F_\sigma (\tfrac{1}{r} P', \rho)
\le
F_\sigma (P', \bar \rho) \, 
\end{equation}
(observe that the previous inequality is not an {\em equality}: while $F_\sigma+\tfrac{1}{2}$ is $1$-homogeneous, $F_\sigma$ is not). 

Observe that $\bar \rho$ might not belong to $\mathcal{R} (P')$ because some radii $\bar\rho (p)$ might exceed $1$. These radii are however no larger than $r$ and their number does not exceed $CL^2$, because the disks $B_{\rho (p)} (p)$ are all contained in $B_{L+1} (O)$ and are pairwise disjoint.
We then define the function $\tilde\rho(p):=\min\{\bar{\rho} (p),1\}$, which is in $\mathcal{R} (P')$, and deduce 
\begin{equation}\label{eq:local55}
    F_\sigma (P', \bar \rho) 
    \leq F_\sigma (P', \tilde \rho)+CL^2(r-1)\,.
\end{equation}
Joining \cref{eq:local53,eq:local54,eq:local55} and taking the supremum over $\rho\in \mathcal R(\tfrac1r P')$, we derive the inequality
\begin{align*}
M_{\sigma'} (\tfrac{1}{r} P')
\le
M_\sigma (\tfrac{1}{r} P') 
\le  M_\sigma (P') + C L^2 \left(r-1\right)
\le  M_\sigma (P) + C L^2 \left(r-1\right)\, . 
\end{align*}
In particular we conclude
\[
\InfoneminM_{\sigma'} (L) \leq k\oneminM_\sigma + C L^2 \left(r-1\right)\, .
\]
Letting $k\uparrow \infty$ we thus reach
\[
\InfoneminM_{\sigma'} (L) \leq \lim_{k\uparrow \infty} k\oneminM_\sigma + C L^2 \left(r-1\right)\, .
\]
This inequality is however valid for every $1<r<\frac{\sigma}{\sigma'}$, in particular we can let $r\downarrow 1$ to reach
\[
\InfoneminM_{\sigma'} (L) \leq  \lim_{k\uparrow \infty} k\oneminM_\sigma .
\]
Finally, letting first $L\uparrow \infty$ and then $\sigma' \uparrow \sigma$ we conclude the proof.
\end{proof}

\subsection{Lipschitz continuity of \texorpdfstring{$M_\sigma$}{Msigma}}
Finally, let us show some form of Lipschitz continuity for $M_\sigma$. Observe that the previous statements of this section were qualitative in nature. The constants appearing in the statements are not explicit and the proofs do not even try to obtain the best possible constants.  
On the other hand, the next one is quantitative and rather sharp. We need this additional care because this lemma will be used in the computer-assisted proof of \cref{thm:0.7} in \cref{s:0.7} and the actual value of the Lipschitz constant matters to ensure the feasibility of the computations.
\begin{lemma}\label{lem:M-lip}
    For $\sigma\in[\tfrac12,1]$, the map $P\mapsto M_\sigma(P)$ satisfies the following Lipschitz bound over sets $P$ and $Q$ with the same cardinality:
    \begin{equation}\label{e:M-lip}
        |M_\sigma(P)-M_\sigma(Q)| 
        \le 
        \inf_{\gamma:P\to Q \text{ bijection}} 
            \sum_{p\in P} |p-\gamma(p)|\, .
    \end{equation}
\end{lemma}
\begin{proof}
    Fix a family of radii $r\in \mathcal R(P)$ and a bijection $\gamma:P\to Q$. We will show below that
    \begin{equation}\label{e:reduced}
        M_\sigma(Q) \ge F_\sigma(P, r) - \sum_{p\in P} |p-\gamma(p)|.
    \end{equation}
    Taking the supremum over $r$ and the infimum over $\gamma$ we reach the inequality \eqref{e:M-lip} without the absolute value on the left hand side. But by symmetry in $P$ and $Q$ this suffices. 
    
    In order to show \eqref{e:reduced} we define $r':Q\to[0,1]$ as $r'(\gamma(p)) := \max\{0, r(p)-|p-\gamma(p)|\}$. 
    It is then clear that $B_{r'(\gamma(p))} (\gamma(p)) \subseteq B_{r(p)} (p)$ for all $p\in P$, which implies $r' \in \mathcal{R} (Q)$, but also $R(Q, r') \leq R(P, r)$ and $\diam\, (U (Q, r')) \leq \diam\, (U (P, r))$. Thus we have 
    \begin{align*}
    M_\sigma(Q) &\geq F_\sigma(Q, r') = \sum_{p\in P} r'(\gamma(p)) - 
    \tfrac1{2\sigma}\min\{R(Q, r') + \tfrac{1}{2}, \diam\, (U (Q, r'))\}\\
    &\geq \sum_{p\in P} \left( r(p) - |p-\gamma(p)|\right) 
    - \tfrac1{2\sigma} \min\{R(P, r) + \tfrac{1}{2}, \diam\, (U (P, r))\}
    \\
    &=
    F_\sigma(P, r) - \sum_{p\in P} |p-\gamma(p)|\,. \qedhere
    \end{align*}
\end{proof}

\section{Proof of \texorpdfstring{\cref{thm:infinite-minmax}}{Theorem~\ref{thm:infinite-minmax}}}\label{sec:proof-infinite-minmax}

The proof of  \cref{thm:infinite-minmax} builds upon the approach developed in \cite{PT}. 
We start recalling some definitions and statements from the latter paper which will be useful for us, first of all
the following lemma, whose neat proof we include for the reader's convenience. In what follows a ``measure'' on a metric space $(X,d)$ will always mean a Borel regular metric outer measure, cf \cite[page 280, Remark (i)]{PT}.
\begin{lemma}[{\cite[Lemma 3]{PT}}]\label{l:furbata}
Let $\mu$ be a Borel measure over a metric space $(X, d)$ satisfying $\mu (S)\leq \diam\, (S)$ for every nonempty Borel set $S\subset X$. Let $E_1, E_2$ be two disjoint Borel subsets of $X$ and let $p_1\neq p_2$ be two points.
For at least one $i\in\{1,2\}$, the following must hold:
\begin{align}
\mu (B_r (p_i) \cap E_i) \leq r+\tfrac{d(p_1, p_2)}{2} \qquad \forall r\geq 0\, .\label{e:furbata}
\end{align}
\end{lemma}
\begin{proof} Assume by contradiction that $B_{r_i} (p_i)$, for $i=1,2$, are two balls for which \cref{e:furbata} fails.
Set $U := B_{r_1} (p_1) \cup B_{r_2} (p_2)$ and observe that 
\begin{align*}
    \mu(U) &\ge \mu(U\cap B_{r_1}(p_1)) + \mu(U\cap B_{r_2}(p_2)) \\
    & > r_1 + \frac{d(p_1, p_2)}{2} + r_2 + \frac{d(p_1, p_2)}{2}  
    = r_1 + r_2 + d(p_1, p_2) \ge \diam(U),
\end{align*}
contradicting $\mu (U) \leq \diam\, (U)$.
\end{proof}

Then, let us recall the Besicovitch pair condition from \cite{PT}. In what follows $\mathcal{B} (X)$ will denote the family of Borel sets of $X$. 
\begin{definition}[Besicovitch Pair Condition~{\cite[Definition 2]{PT}}]\label{def:BPC-property}
    Fix $\sigma\in[\frac12,1]$. 
    A metric space $(X, d)$ is said to satisfy the \emph{Besicovitch pair condition with parameter $\sigma$} if the following holds.

    Given a Borel measure $\mu:\mathcal B(X)\to[0,\infty]$ satisfying $\mu(S) \leq \diam\, (S)$ for every Borel set $S\neq \emptyset$, there is $\tau = \tau(\mu) > 0$ such that, for every $\lambda >0$, there is a $\delta =\delta (\mu, \lambda) > 0$ with the following property. If $E_1$ and $E_2$ are nonempty Borel sets such that 
    \begin{enumerate}[label=(\roman*)]
        \item $0 < \dist(E_1, E_2) < \delta$, where $\dist(E_1, E_2):=\inf \{d (x_1, x_2) : x_i\in E_i\}$ and
        \item $\mu (B_r (x)) > 2\sigma r$ for every $x\in E_1\cup E_2$ and every $0<r<\lambda$,
    \end{enumerate}
    then there is an open set $U\subseteq X$ that intersects both $E_1$ and $E_2$ such that
    \begin{equation*}
    \mu (U\setminus (E_1\cup E_2)) > \tau \diam(U)\, .
    \end{equation*}
\end{definition}

\begin{proposition}[{\cite[Proposition at page 285]{PT}}]\label{prop:pt_pair_condition}
    Fix $\sigma\in[\frac12,1]$.
    If a metric space $(X, d)$ satisfies the Besicovitch pair condition with parameter $\sigma$, then $\bar\sigma(X,d)\le \sigma$.
\end{proposition}

For technical reasons, we will need the following variant of the Besicovitch pair condition.

\begin{definition}[Compact Besicovitch Pair Condition]\label{def:compact-BPC-property}
    Fix $\sigma\in[\frac12,1]$. 
    A metric space $(X, d)$ is said to satisfy the \emph{compact Besicovitch pair condition with parameter $\sigma$} if the following holds.

    Given a Borel measure $\mu:\mathcal B(X)\to[0,\infty]$ satisfying $\mu(S) \leq \diam\, (S)$ for every Borel set $S\neq \emptyset$, there is $\tau = \tau(\mu) > 0$ such that, for every $\lambda >0$, there is a $\delta =\delta (\mu, \lambda) > 0$ with the following property. If $E_1$ and $E_2$ are nonempty compact sets such that 
    \begin{enumerate}[label=(\roman*)]
        \item $0 < \dist(E_1, E_2) < \delta$, where $\dist(E_1, E_2):=\inf \{d (x_1, x_2) : x_i\in E_i\}$ and
        \item $\mu (B_r (x)) > 2\sigma r$ for every $x\in E_1\cup E_2$ and every $0<r<\lambda$,
    \end{enumerate}
    then there is a ball $B_r(x)\subseteq X$ that intersects both $E_1$ and $E_2$ such that
    \begin{equation*}
    \mu (B_r(x)\setminus (E_1\cup E_2)) > \tau r\, .
    \end{equation*}
\end{definition}
There are two main differences compared to the original Besicovitch pair condition: we assume that $E_1,E_2$ are compact instead of just Borel and we require to find a ball instead of an arbitrary open set. Observe also that the inequality required by the Compact Besicovitch Pair Condition, which  is $\mu (B_r(x)\setminus (E_1\cup E_2))>\tau r$, does not involve the diameter of $B_r(x)$: of course it implies $\mu (B_r(x)\setminus (E_1\cup E_2))>\tfrac\tau2 {\diam}\, (B_r(x))$, but in a general metric space the diameter of $B_r (x)$ might be smaller than $2r$.
We prove that also this condition implies an upper bound for $\bar\sigma$, as long as $(X,d)$ is assumed to be complete and separable. Since in our applications we can without loss of generality pass to the completion of a separable space, we pay a small price in terms of generality at this point, but we gain in clarity of the proofs, while at the same time the final outcome will not be affected.

\begin{proposition}\label{prop:compact_pair_condition}
    Fix $\sigma\in[\frac12,1]$.
    If a complete metric space $(X, d)$ satisfies the compact Besicovitch pair condition with parameter $\sigma$, then $\bar\sigma(X,d)\le \sigma$.
\end{proposition}
\begin{proof}
    First of all we wish to show that, without loss of generality, we can assume that $(X,d)$ is separable. Indeed consider a Borel set $E$ in $X$ such that $\mathcal{H}^1 (E)<\infty$. Then $E$ is necessarily separable. Its closure $\overline{E}$ is then a separable metric space. On the other hand, if $(X,d)$ satisfies the compact Besicovitch pair condition with parameter $\sigma$, so does every closed subset of $X$.
    
    From now on we assume therefore that $X$ is also separable. Fix an origin $O\in X$. 

    For an arbitrary set $E\subseteq X$, for any $R>0$, the lower $\mathcal H^1$-densities at any $x\in B_R(O)$ of $E$ and $E\cap B_R(O)$ are equal. Moreover, the set $E$ is $1$-rectifiable if and only if the sets $E\cap B_R(O)$ are rectifiable for all $R>0$. These two observations imply
    \begin{equation*}
        \bar\sigma(X, d) = \sup_{R>0} \bar\sigma(B_R(O), d).
    \end{equation*}
    Observe also that if the compact Besicovitch pair condition with parameter $\sigma$ holds for $X$ then it holds also for $B_R(O)$.

    In view of what we said in the previous paragraph, it is sufficient to show the statement under the additional assumption that $X$ is bounded (as we can then apply it to all balls $B_R(O)$ to recover the result when $X$ is unbounded). When $X$ is bounded we prove that the \emph{compact} Besicovitch pair condition implies the original Besicovitch pair condition of Preiss and Ti\v{s}er. This is sufficient to conclude because of \cref{prop:pt_pair_condition}.

    Assume that the compact Besicovitch pair condition with parameter $\sigma$ holds. Let $\mu$ be a measure so that $\mu(S)\le \diam(S)$. Observe that, since $X$ is bounded, we deduce that $\mu$ is finite. In particular we can conclude that $\mu$ is a Radon measure because $X$ is a complete separable metric space, see e.g. \cite[Theorem V.5.3]{Jacobs}. We need however a stronger property than the usual inner regularity of Radon measures, namely that $\mu (E)=\sup \{\mu (K): K\subset E \mbox{ compact}\}$ for every {\em Borel} set $E$. The latter is also true in complete separable metric spaces for finite Radon measures, see e.g. \cite{Parthasarathy}.

    Assume that the condition holds for $\tau, \lambda, \delta$; we prove that the original Besicovitch pair condition holds for $\tau/4, \lambda, \delta$.

    Take two Borel sets $E_1, E_2\subseteq X$ that satisfy the properties \textit{(i)} and \textit{(ii)} of the Besicovitch pair condition. Let $\tilde E_1\subseteq E_1, \tilde E_2\subseteq E_2$ be two compact subsets such that $\dist(\tilde E_1, \tilde E_2)<\delta$ and 
    \begin{equation*}
        \max\big\{\mu(E_1\setminus\tilde E_1),
                  \mu(E_2\setminus\tilde E_2)\big\}
        \le
        \frac{\tau\dist(E_1, E_2)}{8}.
    \end{equation*}
    Then, $\tilde E_1, \tilde E_2$ satisfy the properties \textit{(i}) and \textit{(ii)} of the \emph{compact} Besicovitch pair condition. Hence, by assumption, we can find $B_r(x)$ intersecting both $\tilde E_1$ and $\tilde E_2$ so that $\mu(B_r(x)\setminus(\tilde E_1\cup \tilde E_2))>\tau r$. Observe that $\diam(B_r(x)) \ge \dist(\tilde E_1, \tilde E_2)\ge \dist(E_1, E_2)$.
    Then,
    \begin{align*}
        \mu(B_r(x)\setminus(E_1\cup E_2))
        &>
        \tau r - \mu(E_1\setminus\tilde E_1) - \mu(E_2\setminus\tilde E_2) 
        \ge \frac\tau2 \diam(B_r(x)) - \frac{\tau\dist(E_1, E_2)}4 \\
        &\ge \frac\tau4\diam(B_r(x)).
    \end{align*}
    By setting $U=B_r(x)$ we deduce the validity of the Besicovitch pair condition. 
\end{proof}

We are ready to prove the following technical statement, which will imply \cref{thm:infinite-minmax} immediately.

\begin{lemma}\label{lem:ubertechno}
    Fix $\tfrac12\le \sigma''<\sigma\le 1$ and $1<L'$ so that $\InfminM_{\sigma''}(L') > 0$. Then the compact Besicovitch pair condition (see \cref{def:compact-BPC-property}) holds for $X=\R^2$ and $\sigma$. Moreover, there is $\eps=\eps(\sigma,\sigma'', L')>0$ so that, for every measure $\mu$ and for every $\lambda>0$, the parameters $\tau$ and $\delta$ in \cref{def:compact-BPC-property} can be chosen to be $\tau:= \frac{\eps}{L+2}$, and $\delta := \lambda$.
\end{lemma}
\begin{proof}
    Fix $\sigma''<\sigma'<\sigma$. Observe that $\eps$ is allowed to depend also on $\sigma'$.
    The value of $\eps$ has to be chosen small enough along the way so that all the steps of the proof work; this does not generate any difficulty.
    
    \vspace{5pt}\noindent\textbf{Step 1:} Setup.\\
    Consider a Borel measure $\mu:\mathcal B(\R^2)\to[0,\infty]$ and two compact sets $E_1, E_2$ satisfying the constraints of \cref{def:compact-BPC-property} (with the parameters mentioned in the statement).
    Assume, by contradiction, that the compact Besicovitch pair condition does not hold for them.

    Without loss of generality (by the translational and rotational invariance of the problem), we may assume that $O\in E_1$, $(-1, 0)\in E_2$ and $\dist(E_1, E_2)=1$. In particular, we deduce $\lambda > 1$ (as $1=\dist(E_1, E_2) < \delta=\lambda$) and so
    $\mu(B_r(x))>2\sigma r$ for all $x\in E_1$ and $0<r\le 1$.
    
    Applying \cref{l:furbata}, we may also assume, without loss of generality, that
    \begin{equation} \label{eq:furbata-application}
        \mu(B_R(O) \cap E_1) \le R + \frac{1}2\quad\text{for all $R>0$.}
    \end{equation}

    \vspace{5pt}\noindent\textbf{Step 2:} Lower bound on $\mu(B_r(x)\cap E_1)$ for $x\in E_1\cap B_{L+1}(O)$.\\
    Fix $L:=L'+1$. Given $x\in E_1\cap B_{L+1}(O)$ and $0<r\le 1$, we have
    \begin{equation*}
        \mu(B_{L+2}(O)\setminus(E_1\cup E_2)) \ge
        \mu(B_r(x)\setminus E_1)
        =
        \mu(B_r(x)) - \mu(B_r(x)\cap E_1)
        > 2\sigma r - \mu(B_r(x)\cap E_1).
    \end{equation*}
    Observe that $B_{L+2}(O)$ intersects both $E_1$ and $E_2$ because $O\in E_1$ and $(-1, 0)\in E_2$. Therefore, since we assume that the compact Besicovitch pair condition is failing, it must be that
    \begin{equation*}
        \mu(B_{L+2}(O)\setminus(E_1\cup E_2)) 
        \le \tau(L+2) = \eps.
    \end{equation*}
    Combining the last two inequalities we obtain the crucial
    \begin{equation}\label{eq:control-mu-b-e}
        \mu(B_r(x)\cap E_1) > 2\sigma r - \eps \quad\text{for all $x\in E_1\cap B_{L+1}(O)$ and $0<r\le 1$.}
    \end{equation}

    \vspace{5pt}\noindent\textbf{Step 3:} The set $E_1$ is $(\sigma', L)$-stable.\\
    The set $E_1$ is $(\sigma', L)$-stable if and only if, for all $x\in E_1\cap B_L(O))$ and all $\tfrac 1L\le r\le 1$, we have $\diam(E_1\cap B_r(x)) \ge 2\sigma'r$. 
    Thanks to \cref{eq:control-mu-b-e} and by the assumption on $\mu$ (i.e., $\mu(S)\le \diam(S)$ for all Borel sets $S$) we know
    \begin{equation}
        \diam(E_1\cap B_r(x))
        \ge 
        \mu(E_1\cap B_r(x))
        \ge 
        2\sigma r - \eps
        \ge
        2\sigma' r,
    \end{equation}
    where the last inequality holds if $\eps$ is chosen small enough (recall that we assume $r\ge \tfrac 1L$). Hence we have shown that $E_1$ is $(\sigma', L)$-stable.

    \vspace{5pt}\noindent\textbf{Step 4:} $M_\sigma(P)\le \frac{\eps}{2\sigma}|P|$ for $P\subseteq E_1\cap B_{L+1}(O)$.\\
    Let $P\subseteq E_1\cap B_{L+1}(O)$ be a finite set and take a family of radii $r\in \mathcal R(P)$. 
    In order to bound $F_\sigma(P, r)$, observe that (see \cref{subsec:objective} for the definition of $U(P, r)$)
    \begin{equation*}
        \mu(U(P, r)) = \sum_{p\in P} \mu(B_{r(p)}(p)\cap E_1)
        \ge 
        \sum_{p\in P} (2\sigma r (p) - \eps)
        \ge 
        2\sigma \Big(\sum_{p\in P} r (p)\Big) - \eps|P|,
    \end{equation*}
    where we have used \cref{eq:control-mu-b-e}.

    On the other hand, we know that $\mu(U(P, r)) \le \diam(U(P, r))$ and also $\mu(U(P, r)) \le \mu(B_{R(P, r)}(O)) \le R(P, r)+\frac12$ thanks to \cref{eq:furbata-application}. Thus, we get
    \begin{equation*}
        \min\Big\{
            \diam(U(P, r)),
            R(P,r)+\frac12
        \Big\}
        \ge
        2\sigma \Big(\sum_{p\in P} r (p)\Big) - \eps|P|,
    \end{equation*}
    which implies 
    \begin{equation*}
        F_\sigma(P, r) \le \frac{\eps}{2\sigma}|P|.
    \end{equation*}
    Since $r\in\mathcal R(P)$ was chosen arbitrarily, we deduce
    \begin{equation}\label{eq:control-M-E1}
        M_\sigma(P) \le \frac{\eps}{2\sigma}|P|\quad\text{for all finite subsets $P\subseteq E_1\cap B_{L+1}(O)$.}
    \end{equation}

    \vspace{5pt}\noindent\textbf{Step 5:} Finding a contradiction.\\
    At this point, the idea shall be clear. The set $E_1$ is a $(\sigma', L)$-stable set with $M_\sigma(E_1)$ small and should be in contradiction with $\InfminM_{\sigma''}(L')>0$. 
    Though, there is an issue that needs to be taken care of. The value of $M_\sigma(E_1)$ could actually be large because our estimate \cref{eq:control-M-E1} works only for finite subsets.
    Let us see how to fix this issue and find the sought contradiction.
    
    Applying \cref{lem:sigma-stable-is-finite} we find $O\in P\subseteq E_1\cap B_{L+1}(O)$ that is $(\sigma'', L')$-stable with $|P|\le N_{\ref*{lem:sigma-stable-is-finite}}(\sigma',\sigma'', L, L')$. 
    Thanks to \cref{eq:control-M-E1}, we have
    \begin{equation}\label{eq:finished-1}
        M_{\sigma''}(P) \le 
        M_{\sigma}(P)
        \le 
        \frac{\eps}{2\sigma}N_{\ref*{lem:sigma-stable-is-finite}}(\sigma', \sigma'',L, L').
    \end{equation}
    Since $O\in P$ is a $(\sigma'', L')$-stable subset, we have
    \begin{equation}\label{eq:finished-2}
        M_{\sigma''}(P) \ge \InfminM_{\sigma''}(L')>0,
    \end{equation}
    where the positivity of the right-hand side is one of the initial assumptions.
    The two inequalities \cref{eq:finished-1,eq:finished-2} yield the contradiction by choosing $\eps$ sufficiently small.
\end{proof}

The proof of \cref{thm:infinite-minmax} now is immediate.
\begin{proof}[Proof of \cref{thm:infinite-minmax}]
    Thanks to \cref{lem:ubertechno}, we know that $\R^2$ satisfies the compact Besicovitch pair condition for all $\tilde\sigma>\sigma$. Thus, thanks to \cref{prop:compact_pair_condition}, we deduce $\bar\sigma\le \tilde\sigma$ for all $\tilde\sigma>\sigma$ and so $\bar\sigma\le \sigma$ as desired.
\end{proof}

\begin{remark}\label{r:mandorla2}
    As a byproduct of the proof of \cref{thm:infinite-minmax}, in the definitions of $\InfminM_\sigma(L)$ we might also require 
    \begin{equation}\label{e:mandorla}
    P\cap B_1((-1,0)) = \emptyset\, ,
    \end{equation}
    reducing the space of configurations and hence potentially increasing the value, as pointed out in \cref{r:mandorla}.
    In an early version of this note we had the constraint \eqref{e:mandorla} everywhere. We decided to drop it because:
    \begin{itemize}
        \item It would not change the values $0\oneminM_\sigma$ and $1\oneminM_\sigma$.
        \item We do not use it to show that $2\oneminM_{0.7}>0$.
        \item It does not improve the best upper bound, independent of $n$, that we can prove for $\bar\sigma(\R^n)$.
        \item It cannot improve a ``universal'' bound for $\bar\sigma(X, d)$, since we can always augment a space $(X,d)$ with a particular configuration $P$ by adding a ball which does not intersect $P$. 
        \item The families $P$ used in \cref{sec:explicit-counterexamples} to show $\InfoneminM_{0.64}=0$ and $\InfoneminM_{0.683}=0$ satisfy \eqref{e:mandorla}.
        \item If we add \eqref{e:mandorla} to the formulations of the infinite-dimensional $\min$-$\max$ problems, we are not able to prove the counterpart of \cref{thm:convergence-finite-minmax} (because \eqref{e:mandorla} is not scaling-invariant, see~\cref{lem:finite-has-stable-subset}). 
    \end{itemize}
\end{remark}

\section{Linear programming for the ``max'' problem}\label{sec:linear-programming}

The goal of this section is to reformulate the maximization problem leading to $M_\sigma(P)$ as a collection of linear programming problems~\cite{Schrijver1986}. This new perspective will play a major role in \cref{sec:0.726,sec:explicit-counterexamples,s:0.7}. 

We will unroll the definition of $F_\sigma(P, r)$ (and consequently of $M_\sigma(P)$) until we get something that is patently a linear programming problem. On our way, we introduce some notation which will be used extensively in the following sections. See \cref{subsec:objective} for the definitions of $F_\sigma$ and $M_\sigma$.

Fix $\sigma\in[\tfrac12, 1]$ and a \emph{finite} set $P\subseteq\R^2$.

Observe that, for any $r\in\mathcal R(P)$, $F_\sigma(P, r) = \max\{F_\sigma^\sharp(P, r), F_\sigma^\flat(P, r)\}$ where
\begin{equation*}\begin{aligned}
    F_\sigma^\sharp(P, r) &:=
    \sum_{p\in P} r(p) - \tfrac1{2\sigma}\big(\tfrac12 + R(P, r)\big), \\
    F_\sigma^\flat(P, r) &:=
    \sum_{p\in P} r(p) - \tfrac{1}{2\sigma}\, \diam(U(P, r))\, .
\end{aligned}\end{equation*}
We introduce the family of radii 
\[
\mathcal{R}^+ (P) := \{r \in \mathcal{R} (P) :\, r(p)>0\text{ for all $p\in P$}\}\, .
\]
It is simple to see that $r\in \mathcal{R}^+ (P)$ if and only if 
\[
\left\{
\begin{array}{ll}
0 < r(p) \leq 1 \qquad &\forall\, p\in P,\\
r(p)+r(p') \leq |p-p'| \qquad & \forall\, p,p'\in P \text{ distinct.}
\end{array}\right.
\]
On the other hand we also can see that, still under the assumption that $r\in \mathcal{R}^+ (P)$, 
\begin{align*}
R(P, r) &= \max_{p\in P} (|p| + r(p)) \,,\\
\diam\, (U (P,r)) &= \max_{p, p'\in P} (|p-p'|+ r(p)+r(p'))\, .
\end{align*}
This motivates the introduction of the two functions
\begin{equation*}\begin{aligned}
    F_\sigma^{+,\sharp}(P, r) &:= 
    \sum_{p\in P} r(p) - \tfrac1{2\sigma} \big(\tfrac12 + \max_{p\in P} |p| + r(p) \big) \, ,\\
    F_\sigma^{+,\flat}(P, r) &:= 
    \sum_{p\in P} r(p) - \tfrac{1}{2\sigma} \max\limits_{p, p'\in P} |p-p'| + r(p) + r(p') \, .
\end{aligned}\end{equation*}
Accordingly, define
\begin{equation*}\begin{aligned}
    M_\sigma^{+,\sharp}(P) &:= 
    \sup_{r\in \mathcal R^+(P)} F_\sigma^{+,\sharp}(P, r) \, ,\\
    M_\sigma^{+,\flat}(P) &:= 
    \sup_{r\in \mathcal R^+(P)} F_\sigma^{+,\flat}(P, r)
\end{aligned}\end{equation*}
and notice that
\begin{equation}\label{eq:def_M_sharp_flat}\begin{aligned}
    M_\sigma^{\sharp}(P) &:=
    \max_{r\in\mathcal R(P)} F_\sigma^\sharp(P, r)
    =
    \max_{P'\subseteq P} M_\sigma^{+,\sharp}(P') , \\
    M_\sigma^{\flat}(P) &:=
    \max_{r\in\mathcal R(P)} F_\sigma^\flat(P, r)
    =
    \max_{P'\subseteq P} M_\sigma^{+,\flat}(P') \, .
\end{aligned}\end{equation}
As an immediate consequence of the definition of $M_\sigma^\sharp$ and $M_\sigma^\flat$, we have
\begin{equation}\label{eq:max_M_sharp_flat}
    M_\sigma(P) = \max\big\{M_\sigma^{\sharp}(P), M_\sigma^{\flat}(P)\big\}.
\end{equation}
It remains to interpret $M_\sigma^{+,\sharp}(P)$ and $M_\sigma^{+,\flat}$ as linear programming problems.
To this purpose, define the two affine functions (with respect to $r\in \R^{|P|}$) 
\begin{align*}
L^\sharp_{\sigma} (P, \bar p)[r] 
    &:= \sum_{p\in P} r(p) - \tfrac1{2\sigma}\big(|\bar p| + r(\bar p) + \tfrac{1}{2}\big) \,,\\
L^\flat_{\sigma} (P, \bar p, \bar p')[r] 
    &:= \sum_{p\in P} r(p) - \tfrac{1}{2\sigma}\big(|\bar p-\bar p'| + r(\bar p) + r(\bar p')\big)\,,
\end{align*}
and the two closed convex polytopes
\begin{equation*}
\mathcal{R}^\sharp_\sigma (P, \bar p) := \left\{
\begin{array}{ll}
r\in [0, 1]^{|P|}, \\
r(p)+r(p')\leq |p-p'| \qquad &\forall\, p, p'\in P\text{ distinct}, \\ 
|p| + r(p) \leq |\bar p| + r(\bar p) \qquad &\forall p\in P
\end{array}
\right\}\,,
\end{equation*}
\begin{equation*}
\mathcal{R}^\flat_{\sigma} (P, \bar p, \bar p') :=
\left\{
\begin{array}{ll}
r\in [0, 1]^{|P|}, \\
r(p)+r(p')\leq |p-p'| \qquad &\forall\, p, p'\in P\text{ distinct}, \\ 
|p-p'| + r(p) + r(p') \leq |\bar p - \bar p'| + r(\bar p) + r(\bar p') \; &\forall p, p'\in P
\end{array}
\right\}\,.
\end{equation*}
Observe that
\begin{align*}
M^{+,\sharp}_\sigma(P, \bar p) 
    & := \max \{L^\sharp_{\sigma}(P, \bar p)[r]: r \in \mathcal{R}^\sharp_{\sigma} (P, \bar p)\}, \\
M^{+,\flat}_\sigma(P, \bar p, \bar p') 
    & := \max \{L^\flat_{\sigma}(P, \bar p, \bar p')[r]: r \in \mathcal{R}^\flat_{\sigma} (P, \bar p, \bar p')\}
\end{align*}
are both linear programming problems and 
\begin{equation}\label{eq:Mplus-linear}\begin{aligned}
M^{+,\sharp}_\sigma(P)
&=
\max_{\bar p\in P} M^{+,\sharp}_\sigma(P, \bar p),\, \\    
M^{+,\flat}_\sigma(P)
&=
\max_{\bar p, \bar p'\in P} M^{+,\flat}_\sigma(P, \bar p, \bar p')
\,.
\end{aligned}\end{equation}
This last claim is obvious if we replace the condition $r\in [0, 1]^{|P|}$ with the stricter one $r\in (0, 1]^{|P|}$ in the definition of the domains $\mathcal R_\sigma^\sharp(P, \bar p)$, $\mathcal R_\sigma^\flat(P, \bar p, \bar p')$ and consequently we replace the $\max$ in the definitions of $M^{+,\sharp}_\sigma(P, \bar p)$, $M^{+,\flat}_\sigma(P, \bar p, \bar p')$ with $\sup$. It is however not difficult to see that our formulation is equivalent and has the advantage that it is stated in terms of maxima on compact domains.  

Summing up, joining \cref{eq:def_M_sharp_flat,eq:max_M_sharp_flat,eq:Mplus-linear}, we have shown that computing $M_\sigma(P)$ is equivalent to taking the maximum of a collection of linear programming problems indexed by the subsets of $P$ and by one or two elements of the chosen subset.

\section{Proof of \texorpdfstring{\cref{thm:0.726}}{Theorem~\ref{thm:0.726}}}\label{sec:0.726}

The crucial task is to compute $M_\sigma(P)$ for $P=\{O, p_1, p_2\}$ with the property that $(p_1, p_2)\in \Delta_\sigma(O, 1)$ (indeed this is the general structure of a set in $\mathcal F_\sigma(1)$). Before coming to it, we wish to show that the only information which matters is the relative distances of the three points $O, p_1, p_2$.

Up to rotations (which do not change the value of $M_\sigma$) the points $p_1,p_2$ are determined by the distances $d_1:= |p_1-O|$, $d_2:=|p_2-O|$ and $D:=|p_1-p_2|$. On the other hand for any choice of these three numbers satisfying the triangle inequalities $D\leq d_1+d_2$, $d_1\leq D + d_2$ and $d_2\leq D+d_1$ there are two points $p_1$, and $p_2$ consistent with such distances. Without loss of generality we can relabel $p_1$ and $p_2$ so that $d_2\leq d_1$. In particular the triangle inequalities become $D\leq d_1+d_2$ and  $d_1\leq D+d_2$. In order to belong to $\Delta_\sigma(O, 1)$ we must satisfy the conditions
\begin{equation}\label{e:admiss}
d_2\leq d_1\leq 1 < 2\sigma \le D\leq d_1+d_2\, 
\end{equation}
(observe that since $D\geq 2\sigma \geq 1\geq d_1$ we can ignore the triangle inequality $d_1\leq D+d_2$).  
\cref{f:triangolo} gives an illustration of the configuration of points $\{O,p_1, p_2\}$ and the corresponding triangle. 

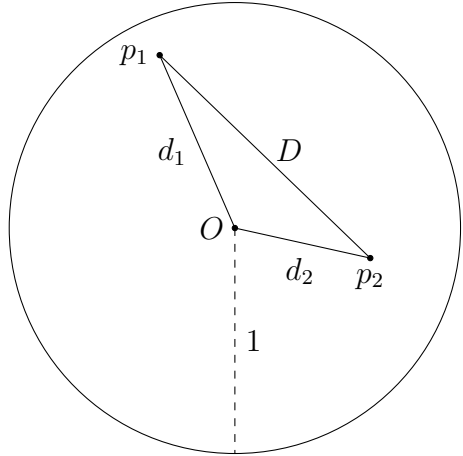
\begin{figure}
\begin{tikzpicture}

\coordinate (O) at (0, 0);
\coordinate (P1) at (-1,2.3);
\coordinate (P2) at (1.8, -0.4);

\filldraw[black] (O) circle (1pt);
\node[left] at (O) {$O$};
\draw (O) circle [radius =3.0];
\filldraw[black] (P1) circle (1pt);
\node[left] at (P1) {$p_1$};

\filldraw [black] (P2) circle (1pt);
\node[below] at (P2) {$p_2$};
\draw (O) -- (P1) -- (P2) -- (O);
\node[above right] at (0.4, 0.75) {$D$};
\node[above left] at (-0.5,0.7) {$d_1$};
\node[below left] at (1.2, -0.25) {$d_2$};
\draw (0,0) [dashed] -- (0,-3);
\node[right] at (0,-1.5) {$1$};
\end{tikzpicture}
\caption{The points $p_1, p_2$ are contained in the unit ball $B_{1}(O)$ and form, together with the origin $O$, a triangle of sides $D \geq 2\sigma$, $d_2$, and $d_1$.}\label{f:triangolo}
\end{figure}

We will compute $M_\sigma(P)$ for $P=\{O, p_1, p_2\}$ in terms of $d_1, d_2, D$ satisfying \cref{e:admiss}. 
In order to accomplish the latter task, we compute separately $M_\sigma^\sharp(P)$ and $M_\sigma^\flat(P)$ (see \cref{eq:def_M_sharp_flat} for the definitions).

\subsection{Computing \texorpdfstring{$M^\sharp_\sigma$}{M-sharp}} 
Let us denote
by $r_0,r_1,r_2$ the radii respectively associated with $O, p_1, p_2$ when $r=(r_0, r_1, r_2)\in \mathcal{R} (P)$.
In this section, we will always consider $2\sigma F^\sharp_\sigma$ instead of $F^\sharp_\sigma$ to avoid denominators.
We compute $M_\sigma^{+,\sharp}(P')$ for all subsets $P'\subseteq P$. To do so, we need the following two simple identities (whose proofs follow from the definition of $M_\sigma^{+,\sharp}$ and $F_\sigma^{+,\sharp}$). 
For any $p\in\R^2$ with $|p|\le 1$, we have
\begin{align}
        2\sigma M_\sigma^{+, \sharp}(\{p\}) &= 2\sigma-\tfrac32 -|p| \,, \label{eq:M_plus_sharp_one}\\
        2\sigma M_\sigma^{+, \sharp}(\{O, p\}) &= (2\sigma-1)|p| -\tfrac12 \, . \label{eq:M_plus_sharp_two}
\end{align}
We are now ready to compute $M_\sigma^{+,\sharp}(P')$ for all the nonempty subsets $P'\subseteq P$.

\vspace{5pt}\noindent\textbf{Case 1: $P'=\{O\}$ or $P'=\{p_1\}$ or $P'=\{p_2\}$.}\\ 
Thanks to \cref{eq:M_plus_sharp_one}, we get that the maximum of $M_\sigma^{+,\sharp}$ over these subsets is achieved by $2\sigma M_\sigma^{+,\sharp}(\{O\}) = 2\sigma-\frac32$.

\vspace{5pt}\noindent\textbf{Case 2: $P'=\{O, p_1\}$ or $P'=\{O, p_2\}$.}\\ 
Thank to \cref{eq:M_plus_sharp_two}, we get that the maximum of $M_\sigma^{+,\sharp}$ over these subsets is achieved by $2\sigma M_\sigma^{+,\sharp}(\{O, p_1\}) = (2\sigma-1)d_1-\tfrac12$. Observe that, since $d_1\le 1$, this is smaller or equal than the value $2\sigma-\frac32$.

\vspace{5pt}\noindent\textbf{Case 3: $P'=\{p_1,p_2\}$.}\\ 
For any $r\in\mathcal R^+(P')$ (or $r\in\mathcal R(P)$), we have
\begin{align*}
    2\sigma F_\sigma^{+,\sharp}(P', r) + \tfrac12 
    & = 
    2\sigma(r_1+r_2) - \max\big\{d_1 + r_1, d_2 + r_2\big\}
    \le 
    2\sigma(r_1+r_2) - \tfrac{d_1+d_2+r_1+r_2}{2}
    \\
    &=
    \big(2\sigma-\tfrac12\big)(r_1+r_2) - \tfrac{d_1+d_2}2
    \le
    \big(2\sigma-\tfrac12\big)D - \tfrac{d_1+d_2}2
    .
\end{align*}
On the other hand if we set $r_1 = \frac{D+d_2-d_1}{2}$ and $r_2 = \frac{D+d_1-d_2}{2}$
the inequality above is attained.
In order for this choice of radii to be admissible we need them to satisfy $0<r_1, r_2 \le 1$, which holds as a consequence of $\max\{d_1,d_2\} < D \le d_1+d_2$.
Therefore, in this case,
\begin{equation*}
2\sigma M^{+,\sharp}_\sigma(P') = 2\sigma D - \tfrac{D+d_1+d_2}{2}- \tfrac{1}{2}\, . 
\end{equation*}
Note that the maximum is achieved when the three circles $B_{r_1} (p_1)$, $B_{r_2} (p_2)$, and $B_{R(P, r)} (O)$ are all tangent to each other, cf. \cref{f:apollonio-1}.

\begin{figure}
\begin{tikzpicture}
\filldraw[black] (0,0) circle (1pt);
\node[left] at (0,0) {$O$};
\filldraw[black] (1,2) circle (1pt);
\node[left] at (1,2) {$p_1$};
\filldraw [black] (2,-0.5) circle (1pt);
\node[below] at (2, -0.5) {$p_2$};
\def\da{sqrt(5)};
\def\db{sqrt(4.25)};
\def\D{sqrt(1+2.5*2.5)};
\def\ra{(\D+\db-\da)/2};
\def\rb{(\D+\da-\db)/2};
\def\R{(\D+\da+\db)/2};
\draw (0,0) circle [radius={\R}];
\draw (1,2) circle [radius={\ra}];
\draw (2,-0.5) circle [radius={\rb}];
\draw[dashed] (0,0) -- (0,{-\R});
\node[left] at (0,{-\R/2}) {$R (P,r)$};
\draw (2,-0.5) -- (0,0) -- (1,2);
\node[above] at (1.4,-.35) {$d_2$};
\node[right] at (0.6,1.2) {$d_1$}; 
\end{tikzpicture}
\caption{The configuration achieving $M^{+,\sharp}_\sigma(\{p_1,p_2\})$. The three disks $B_{r_1} (p_1)$, $B_{r_2} (p_2)$, and $B_{R(P,r)} (O)$ are all tangent to each other.}\label{f:apollonio-1}
\end{figure}
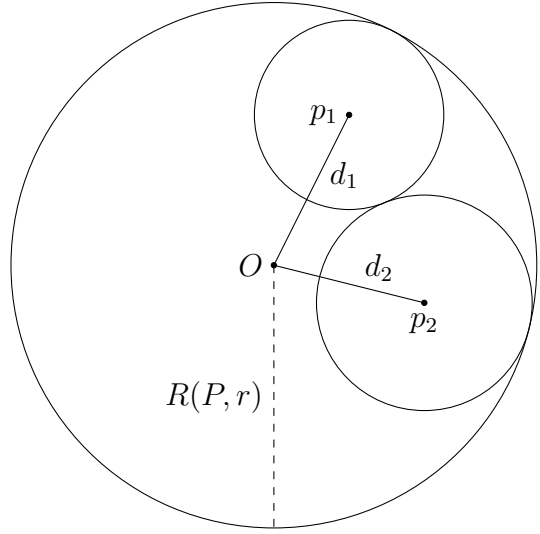

\vspace{5pt}\noindent\textbf{Case 4: $P'=P=\{O, p_1,p_2\}$.}\\
For any $r\in\mathcal R^+(P)$ (or $r\in\mathcal R(P)$), we have
\begin{align*}
2\sigma F^{+,\sharp}_\sigma(P, r) + \tfrac12 
&= 2\sigma (r_0+r_1+r_2) - \max \{r_1+d_1, r_2+d_2\}
\\
&\le 2\sigma (r_0+r_1+r_2) - d_1 - r_1 \\
&= (\sigma+ \tfrac{1}{2}) (r_0+r_2) + (\sigma-\tfrac{1}{2}) (r_1+r_0) 
+ (\sigma-\tfrac{1}{2}) (r_1+r_2) - d_1\\
&\leq (\sigma+\tfrac{1}{2}) d_2 + (\sigma-\tfrac{1}{2}) d_1 + (\sigma-\tfrac{1}{2}) D
-d_1 .
\end{align*} 
In particular we conclude
\begin{equation*}
2\sigma M^{+,\sharp}_\sigma(P, r)\leq \sigma (D+d_1+d_2) - \tfrac{D+3d_1-d_2}{2} - \tfrac{1}{2}\, .
\end{equation*}
On the other hand, the equality is achieved if we set $r_0=\frac{d_1+d_2-D}{2}$, $r_1=\frac{D+d_1-d_2}{2}$, and $r_2=\frac{D+d_2-d_1}{2}$. Indeed, all these radii are nonnegative because of the triangle inequalities and they satisfy $r_0+r_1= d_1\leq 1$, $r_0+r_2=d_2\leq 1$, and $r_1+r_2=D$, thus guaranteeing that they satisfy all the constraints. Finally, since $r_1\geq r_2$ and $d_1\geq d_2$, we have $\max \{d_1+r_1, d_2+r_2\}=d_1+r_1$. 
Note that, as observed at the end of \cref{sec:linear-programming}, we can replace the condition $r\in\mathcal R^+(P)$ with $r\in\overline {\mathcal R(P)} = \mathcal R(P)$ in the definition of $M^{+,\sharp}_\sigma$ and therefore it is not an issue if the radii $r_0, r_1, r_2$ are not strictly positive.
Hence,
\begin{equation*}
2\sigma M^{+,\sharp}_\sigma(P, r) = \sigma (D+d_1+d_2) - \tfrac{D+3d_1-d_2}{2} - \tfrac{1}{2}\, .
\end{equation*}
Observe in passing that this maximum is achieved when the three circles $B_{r_0} (O)$, $B_{r_1} (p_1)$ and $B_{r_2} (p_2)$ are all tangent to each other, cf. \cref{f:apollonio-2}. 

Summarizing the four cases analyzed, we have shown
\begin{equation}\label{e:max-diesis}
2\sigma M^\sharp_\sigma(P) = \max \left\{ 2\sigma - 1, 2\sigma D - \tfrac{D+d_1+d_2}{2}, \sigma (D+d_1+d_2) - \tfrac{D+3d_1-d_2}{2}\right\} - \tfrac{1}{2}\, .
\end{equation}

\begin{figure}
\begin{tikzpicture}
\filldraw[black] (0,0) circle (1pt);
\node[left] at (0,0) {$O$};
\filldraw[black] (1,2) circle (1pt);
\node[left] at (1,2) {$p_1$};
\filldraw [black] (2,-0.5) circle (1pt);
\node[below] at (2, -0.5) {$p_2$};
\def\da{sqrt(5)};
\def\db{sqrt(4.25)};
\def\D{sqrt(1+2.5*2.5)};
\def\ra{(\D+\da-\db)/2};
\def\rb{(\D+\db-\da)/2};
\def\r{\da-\ra};
\draw (0,0) circle [radius={\r}];
\draw (1,2) circle [radius={\ra}];
\draw (2,-0.5) circle [radius={\rb}];
\draw (2,-0.5) -- (0,0) -- (1,2) -- (2,-0.5);
\node[below] at (1.1,-.25) {$d_2$};
\node[left] at (0.6,1.2) {$d_1$}; 
\node[above] at (1.7,0.75) {$D$};
\end{tikzpicture}
\caption{The maximum value of $F^\sharp_\sigma$ under the assumption $r_0>0$ is achieved when the three disks $B_{r_1} (p_1)$, $B_{r_2} (p_2)$, and $B_{r_0} (O)$ are all tangent to each other.}\label{f:apollonio-2}
\end{figure}
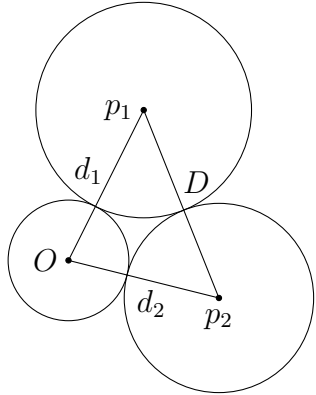

\subsection{Computing \texorpdfstring{$M^\flat_\sigma$}{M-flat}}
We next wish to show that $M^\flat_\sigma(P)=0$ (as in the previous section, we will consider $2\sigma M^\flat_\sigma$ to avoid denominators).
We start by observing that now the function is symmetric in the entries $r_0$, $r_1$, and $r_2$.
We consider $M_\sigma^{+,\flat}(P')$ for all subsets $P'\subseteq P$. To this purpose, we need the following two simple identities (whose proofs follow from the definition of $M_\sigma^{+,\flat}$ and $F_\sigma^{+,\flat}$). For any $p, q\in\R^2$ with $|p-q|\le 2$, we have
\begin{align}
    2\sigma M_\sigma^{+,\flat}(\{p\}) &= 0 \,, \label{eq:M_plus_flat_one} \\
    2\sigma M_\sigma^{+,\flat}(\{p, q\}) &= 2(\sigma-1)|p-q|\le 0 \,. \label{eq:M_plus_flat_two}
\end{align}
We are now ready to study $M_\sigma^{+,\flat}(P')$ for all the nonempty subsets $P'\subseteq P$.

\vspace{5pt}\noindent\textbf{Case 1: $|P'|=1$.}\\ 
Thanks to \cref{eq:M_plus_flat_one}, we get that $2\sigma M_\sigma^{+,\sharp}(P') = 0$.

\vspace{5pt}\noindent\textbf{Case 2: $|P'|=2$.}\\ 
Thanks to \cref{eq:M_plus_flat_two}, we get that $2\sigma M_\sigma^{+,\sharp}(P') \le 0$.

\vspace{5pt}\noindent\textbf{Case 3: $P'=P=\{O, p_1, p_2\}$.}\\ 
For any $r\in\mathcal R^+(P)$ (or $r\in\mathcal R(P)$), we have
\[
2\sigma F_\sigma^{+,\flat}(P, r) = 2\sigma (r_0+r_1+r_2) - \max \{D+r_1+r_2, d_1+r_1+r_0, d_2+r_2+r_0\}\, .
\]
Since $r_0+r_1\le d_1$ and $r_0+r_2\le d_2$, we get $r_0 \le \frac{d_1+d_2-(r_1+r_2)}2$. Thus, we have
\begin{align*}
    2\sigma F_\sigma^\flat(P, r) 
    &\le 
    2\sigma(r_0+r_1+r_2) - (D+r_1+r_2)
    =
    (2\sigma-1)(r_1+r_2) + 2\sigma r_0 - D \\
    &\le
    (2\sigma-1)(r_1+r_2) + \sigma (d_1+d_2 - (r_1+r_2)) - D
    \\
    &=
    (\sigma-1)(r_1+r_2) + \sigma(d_1+d_2) - D
    \le 0 + \sigma(1+1) - 2\sigma = 0.
\end{align*}
We thus conclude that $M_\sigma^\flat(P) = 0$ as desired.

\subsection{Conclusion}
Observe that, since $2\sigma M_\sigma (P) \geq 2\sigma -\frac{3}{2}>0$ when $\sigma > \frac{3}{4}$, we can assume, without loss of generality, that $\sigma \in (\frac{1}{2}, \frac{3}{4}]$.
We have thus arrived at the conclusion that we need to minimize 
\begin{align*}
2\sigma M_\sigma(P) 
= \max \left\{\tfrac{1}{2}, 2\sigma D - \tfrac{D+d_1+d_2}{2}, \sigma (D+d_1+d_2) - \tfrac{D+3d_1-d_2}{2}\right\} - \tfrac{1}{2} =: \mu (D, d_1, d_2) 
\end{align*}
under the assumptions that $\sigma \in (\frac{1}{2}, \frac{3}{4}]$ and that $(p_1, p_2)\in \Delta_\sigma(O, 1)$ with $|p_1|\geq |p_2|$. According to our discussion, this is equivalent to minimize the function $\mu$ for $(D, d_1, d_2)$ varying in the set 
\[
\Delta := \left\{1\geq d_1\geq d_2\geq 0, d_2+d_1\geq D \geq 2\sigma\right\} .
\]
Next consider $(D, d_1, d_2)\in \Delta$ and, if $D>2\sigma$, observe that the triple $(2\sigma, \frac{2\sigma d_1}{D}, \frac{2\sigma d_2}{D})$ still belongs to $\Delta$ and $\mu (2\sigma, \frac{2\sigma d_1}{D}, \frac{2\sigma d_2}{D})\leq \mu (D, d_1, d_2)$. Hence, the minimum of $\mu$ in $\Delta$ is achieved at a triple $(D, d_1, d_2)$ with $D=2\sigma$.

We then introduce the functions
\begin{align*}
\mu_1 (d_1, d_2) &= 4\sigma^2 - \sigma - \tfrac{d_1+d_2}{2}\, ,\\
\mu_2 (d_1, d_2) &= 2\sigma^2 -\sigma
+ (\sigma +\tfrac{1}{2})d_2 - (\tfrac{3}{2} -\sigma) d_1\, , \\
\bar{\mu} (d_1, d_2) &= \max \{\mu_1 (d_1, d_2), \mu_2 (d_1, d_2)\}\, ,
\end{align*}
and the set $\tilde{\Delta} := \{1\geq d_1\geq d_2\geq 0, d_2+d_1\geq 2\sigma\}$. 

So, for the values of $\sigma$ which interest us, namely $\sigma \in (\frac{1}{2}, \frac{3}{4}]$,  
\[
1\oneminM_\sigma = \tfrac1{2\sigma}\max\{0, \min \{\bar{\mu} (d_1, d_2) : (d_1, d_2)\in \tilde{\Delta}\} - \tfrac{1}{2}\}\, .
\]

Hence we are left with the task of finding $\min \{\bar{\mu} (d_1, d_2) : (d_1, d_2)\in \tilde{\Delta}\}$ under the assumption that $\sigma\in (\frac{1}{2}, \frac{3}{4}]$. 
Note that $\partial_1 \mu_1 =-\tfrac12<0$ and $\partial_1\mu_2 = \sigma-\tfrac 32<0$. Hence the minimum of $\bar\mu$ must be achieved at a point $(d_1,d_2)\in\tilde \Delta$ so that $(d_1+\eps,d_2)\not\in\tilde\Delta$ for all $\eps>0$. Hence, it must be $d_1=1$.
So, we have reduced the problem to finding
\begin{equation*}
    \min \big\{\max\{\mu_1(1, d_2), \mu_2(1, d_2)\}:\, 2\sigma-1\le d_2\le 1 \big\}.
\end{equation*}
Observe that $\mu_2 (1,2\sigma-1)<\mu_1 (1, 2\sigma-1)$ and $\mu_1 (1,1) < \mu_2 (1,1)$. 
Therefore, since $\mu_1$ and $\mu_2$ are affine functions and $\partial_2\mu_1<0<\partial_2\mu_2$, the minimum we are looking for is achieved at the point $2\sigma-1\le d_2\le 1$ such that $\mu_1(1, d_2)=\mu_2(1, d_2)$, which yields 
\[
d_2 = \tfrac{2\sigma^2 - \sigma + 1}{\sigma+1}.
\]
We have thus arrived to the formula for $1\oneminM_\sigma$ on the interval $(\frac{1}{2}, \frac{3}{4}]$, which is given by 
\begin{equation}\label{e:final-bound-below}
1\oneminM_\sigma=\tfrac1{2\sigma}\max \left\{0, 4\sigma^2 - \sigma - \tfrac{\sigma^2+1}{\sigma+1}-\tfrac{1}{2}\right\} = \tfrac{\max\{0, 8\sigma^3+4\sigma^2-3\sigma-3\}}{4\sigma(\sigma+1)}.
\end{equation}
Observe that the derivative of $\eta(\sigma) := 8 \sigma^3 + 4 \sigma^2 - 3\sigma-3$ is given by $24 \sigma^2 + 8 \sigma - 3$,
which is positive on $[\frac{1}{2}, \infty)$. Since $\eta (\frac{1}{2})<0$ and $\eta(\frac{3}{4})>0$, the polynomial $\eta$ has a unique zero $\sigma_{PT}$ in $[\frac{1}{2}, \infty)$, which falls in the interval $[\frac{1}{2}, \frac{3}{4}]$. For $\sigma \in [0,\frac{1}{2}]$ clearly $\eta (\sigma)\leq 1+1-0-3<-1$ and so $\sigma_{PT}$ is the only positive real root of $\eta$. In fact, it is not hard to see, using elementary inequalities that $\sigma_{PT}$ is the only {\em real} zero of the polynomial $\eta$, however this is irrelevant for our purposes. Using the Cardano-Tartaglia formula we can then write a closed expression for $\sigma_{PT}$ involving radicals.

\section{Explicit construction of counterexamples}\label{sec:explicit-counterexamples}
We show that $2\oneminM_{\sigma}=0$ for $\sigma=0.683$ and $\InfoneminM_{\sigma}=0$ for the unique positive $\sigma$ solving $32s^3-32s^2+12s-3$ (which are, respectively, one part of the statement of \cref{thm:0.7} and \cref{thm:0.64}) by describing two valid---for the problems at hand---subsets $P$ with $M_\sigma(P)=0$. Similarly, we will show that 
\[
k\oneminM_{\sigma_{PT}}(X, d, O)=0 \qquad \mbox{for all $k\in\N_0$}
\]
in an appropriately constructed metric space, as stated in \cref{thm:0.726-metric}.

\subsection{\texorpdfstring{$2\oneminM_{0.683}=0$}{2-MinMax counterexample}}
Let $\sigma:=0.683$.
We construct a $6$-point set $P$ that belongs to $\mathcal F_\sigma(2)$ such that $M_\sigma(P)=0$.

Let $P:=\{O, a, b, a_1, a_2, b_1, b_2\}$, where
\begin{alignat*}{2}
    &a   = (0.306, 0.952), \quad &&  b = (0.034,-0.387), \\
    &a_1 = O,              \quad &&a_2 = (0.464, 1.815), \\
    &b_1 = (0.031, 0.599), \quad &&b_2 = (0.516, -0.679). 
\end{alignat*}

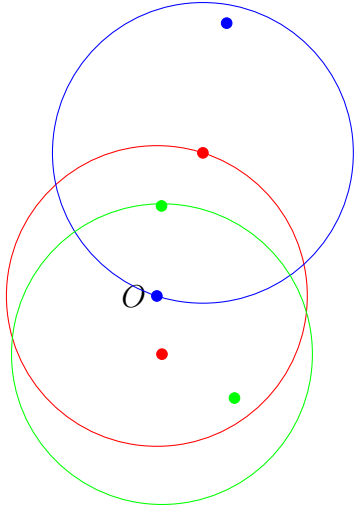
\begin{figure}
\begin{tikzpicture}[scale=2]
\node[left] at (0,0) {$O$};
\filldraw[color=blue] (0,0) circle (1pt);
\draw[color=red] (0,0) circle (1cm);
\filldraw[color=red] (0.306, 0.952) circle (1pt);
\draw[color=blue] (0.306, 0.952) circle (1cm);
\filldraw[color=red] (0.034,-0.387) circle (1pt);
\draw[color=green] (0.034,-0.387) circle (1cm);
\filldraw[color=blue] (0.464, 1.815) circle (1pt);
\filldraw[color=green] (0.031, 0.599) circle (1pt);
\filldraw[color=green] (0.516, -0.679) circle (1pt);
\end{tikzpicture}
\caption{The configuration of $6$ points that proves $2\oneminM_{0.683}=0$.\\
The points $a$ and $b$ are red and are inside the red circle, which is the circle centered at $O$ of radius $1$.
The points $a_1$ and $a_2$ are blue and are inside the blue circle, which is the circle centered at $a$ of radius $1$. Note that $a_1$ coincides with the origin. 
Finally, $b_1$ and $b_2$ are green and are inside the green circle, which is the circle centered at $b$ of radius $1$.
}\label{fig:0.683}
\end{figure}

Observe that $(a, b)\in \Delta_\sigma(O, 1)$, $(a_1, a_2)\in \Delta_\sigma(a, 1)$, and $(b_1, b_2)\in \Delta_\sigma(b, 1)$.
Therefore, $P\in\mathcal F_\sigma(2)$.

It remains to show that $M_\sigma(P)=0$.
Recall that $M_\sigma(P)=\max\{M_\sigma^\sharp(P), M_\sigma^\flat(P)\}$ (as in \cref{eq:max_M_sharp_flat}).
It can be verified numerically that $M_\sigma^\sharp(P)=-0.00032\ldots<0$.
To check $M^\flat_\sigma(P)=\max \{F^\flat_\sigma(P, r): r \in \mathcal{R} (P)\}\leq 0$ we need to be more careful because $F^\flat_\sigma(P, r)=0$ for $r\equiv 0$.

First of all, \cref{eq:M_plus_flat_one} shows that, for any set $Q\subseteq\R^2$ and $r\in\mathcal R(Q)$, we have $F_\sigma^\flat(Q, r)\le 0$ if at most one radius is nonzero.
Hence, we can restrict our attention to the function
\begin{equation}\label{eq:emmebarrata}
\bar{M}^\flat_\sigma(P) := \sup \{F^\flat_\sigma(P, r) : r\in \mathcal{R} (P) 
\textrm{ and at least two radii are nonzero}\}\, .
\end{equation}
The latter can be phrased as the maximum of finitely many linear programming problems as in \cref{sec:linear-programming}; then with the help of a computer we can check $\bar{M}^\flat_\sigma(P) = -0.1803\ldots < 0$.

\begin{remark}\label{r:mandorla3}
Observe that $B_1 (-1,0) \cap P=\emptyset$, in fact the first coordinates of all points are nonnegative. In particular this example complies also with the additional constraint of \cref{r:mandorla,r:mandorla2}.
\end{remark}

\subsection{\texorpdfstring{$\InfoneminM_\sigma=0$}{Infinite min-max} for \texorpdfstring{$\sigma\approx0.64368\ldots$}{sigma=0.64368...}}
Let $\sigma\approx0.64368\ldots$ be the unique positive solution of $32\sigma^3-32\sigma^2+12\sigma-3=0$.
We construct a $4$-point set $P$ containing the origin that is $(\sigma,\infty)$-one-stable and such that $M_\sigma(P)=0$.

\subsubsection{Setup}
Let $P:=\{p_1, p_2, p_3, p_4\}$, where
\begin{equation*}
    p_1=O,\, 
    p_2 = (2-2\sigma, 0),\, 
    p_3 =  \big(2-2\sigma+\cos\theta, -\sin\theta\big),\,
    p_4 = \big(-\cos\theta, -\sin\theta\big) ,\,
\end{equation*}
where $0<\theta<\tfrac{\pi}2$ satisfies $\cos\theta=\frac{8\sigma-5}{4(1-\sigma)}$.
The distances between these points are
\begin{align*}
    &|p_1-p_2| = 2(1-\sigma) := b, \\
    &|p_2-p_3| = |p_4-p_1| = 1,\\
    &|p_1-p_3| = |p_2-p_4| = 2\sigma,\\
    &|p_3-p_4| = 2\left(1-\sigma+\tfrac{8\sigma-5}{4(1-\sigma)}\right) =: B \approx 0.92239\ldots .
\end{align*}
Observe that $P$ is the set of vertices of an isosceles trapezoid whose bases have lengths $b$ and $B$, whose legs have length $1$, and whose diagonals have length $2\sigma$ (see \cref{fig:0.64}).

The set $P$ is $(\sigma,\infty)$-one-stable as its points satisfy
\begin{equation*}
    (p_2, p_4)\in \Delta_\sigma(p_1, 1),\,
    (p_1, p_3)\in \Delta_\sigma(p_2, 1),\,
    (p_2, p_4)\in \Delta_\sigma(p_3, 1),\,
    (p_1, p_3)\in \Delta_\sigma(p_4, 1).
\end{equation*}

\begin{figure}
\begin{tikzpicture}[scale=3.5]

\def\ss{0.64368};
\def\ddx{(8*\ss-5)/(4-4*\ss)};
\def\ddy{sqrt(1-\ddx*\ddx)};

\coordinate (P1) at (0,0);
\coordinate (P2) at ({2-2*\ss}, 0);
\coordinate (P3) at ({2-2*\ss+\ddx}, {0-\ddy});
\coordinate (P4) at ({-\ddx}, {-\ddy});

\filldraw[black] (P1) circle (0.5pt);
\node[left] at (P1) {$p_1=O$};

\filldraw[black] (P2) circle (0.5pt);
\node[right] at (P2) {$p_2$};

\filldraw[black] (P3) circle (0.5pt);
\node[right] at (P3) {$p_3$};

\filldraw[black] (P4) circle (0.5pt);
\node[left] at (P4) {$p_4$};

\draw (P1) -- (P2) -- (P3) -- (P4) -- (P1);

\draw (P1) --  (P3);

\node[above] at ({1-\ss}, 0) {$b$};
\node[below] at ({1-\ss}, {-\ddy}) {$B$};
\node[left] at ({-\ddx/2-0.05}, {-\ddy/2}) {$1$};
\node[right] at ({2-2*\ss+\ddx/2+0.05}, {-\ddy/2}) {$1$};
\node[below] at ({1-\ss}, {-\ddy/2}) {$2\sigma$};

\end{tikzpicture}
\caption{The configuration of four points $p_1, p_2, p_3, p_4$ that proves $\InfoneminM_{\sigma}=0$, where $\sigma\approx 0.64368\dots$ is the unique positive solution of $32\sigma^3-32\sigma^2+12\sigma-3=0$.
}\label{fig:0.64}
\end{figure}
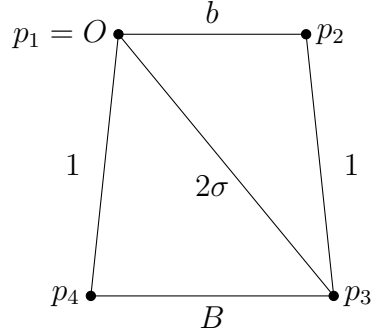

\subsubsection{Proof of $M_\sigma(P)=0$}
It remains to show that $M_\sigma(P)=0$. 
As in the previous subsection, we will check separately that $M_\sigma^\sharp(P)=0$ and $M^\flat_\sigma(P)=0$. 

Let us start with $M^\sharp_\sigma(P)$. Recall that
\begin{equation*}
    M^\sharp_\sigma(P) = \max_{P'\subseteq P} M^{+,\sharp}_\sigma(P').
\end{equation*}
We argue in different ways depending on the size of $P'$.
\begin{itemize}
    \item If $1\le |P'|\le 3$, then we can check with the help of a computer that $M_\sigma^{+,\sharp}(P')\le -0.03014\ldots < 0$. 
    \item If $|P'|=4$, then $P'=P$. In this case, given $r\in\mathcal R(P)$, denote $r_i:= r(p_i)$. We have
    \begin{align*}
        F_\sigma^{+,\sharp}(P,r) &= r_1+r_2+r_3+r_4 - \tfrac1{2\sigma}\big(\tfrac12 + R(P, r)\big),
        \\
        R(P, r) &\ge \max\big\{
            |p_1-p_3| + r_3,
            |p_1-p_4| + r_4
        \big\}
        \ge
        \frac{2\sigma+r_3+1+r_4}2 .
    \end{align*}
    Therefore,
    \begin{equation*}
        F_\sigma^{+,\sharp}(P,r) \le r_1+r_2 + (1-\tfrac1{4\sigma})(r_3+r_4)-\big(\tfrac1{2\sigma}+\tfrac12\big).
    \end{equation*}
    Since $r\in\mathcal R(P)$, we have $r_1+r_2\le |p_1-p_2|=b$ and $r_3+r_4\le |p_3-p_4|=B$. We finally obtain
    \begin{equation*}
        M_\sigma^{+,\sharp}(P) = \sup_{r\in\mathcal R(P)} F_\sigma^{+,\sharp}(P,r) \le b + \big(1-\tfrac1{4\sigma}\big)B  -\big(\tfrac1{2\sigma}+\tfrac12\big) = 
        \frac{32\sigma^3-32\sigma^2+12\sigma-3}{8\sigma(1-\sigma)} = 0.
    \end{equation*}
\end{itemize}
Now, let us consider $M_\sigma^\flat(P)$. Recall that
\begin{equation*}
    M^\flat_\sigma(P) = \max_{P'\subseteq P} M^{+,\flat}_\sigma(P').
\end{equation*}
We argue in different ways depending on the size of $P'$.
\begin{itemize}
    \item If $|P'|\le 1$, then \cref{eq:M_plus_flat_one} tells us that $M^{+,\flat}_\sigma(P')=0$.
    \item If $2\le |P'|\le 3$, then we can check with the help of a computer that $M_\sigma^{+,\flat}(P')\le -0.23604\ldots < 0$.
    \item If $|P'|=4$, then $P'=P$. In this case, given $r\in\mathcal R(P)$, denote $r_i:= r(p_i)$. We have
    \begin{align*}
        F_\sigma^{+,\flat}(P,r) &= r_1+r_2+r_3+r_4 - \tfrac1{2\sigma}\diam(U(P, r)),
        \\
        \diam(U(P, r)) &\ge \max\big\{
            |p_1-p_3| + r_1 + r_3,
            |p_2-p_4| + r_2 + r_4
        \big\}
        \ge
        2\sigma + \tfrac{r_1+r_2 + r_3+r_4}2 .
    \end{align*}
    Therefore,
    \begin{equation*}
        F_\sigma^{+,\flat}(P,r) \le \big(1-\tfrac1{4\sigma}\big)(r_1+r_2+r_3+r_4) - 1.
    \end{equation*}
    Since $r\in\mathcal R(P)$, we have $r_1+r_2\le |p_1-p_2|=b$ and $r_3+r_4\le |p_3-p_4|=B$ and so $r_1+r_2+r_3+r_4\le b+B$. We finally obtain
    \begin{equation*}
        M_\sigma^{+,\flat}(P) = \sup_{r\in\mathcal R(P)} F_\sigma^{+,\flat}(P,r) \le \big(1-\tfrac1{4\sigma}\big)(b+B) - 1 = 
        \frac{32\sigma^3-32\sigma^2+12\sigma-3}{8\sigma(1-\sigma)} = 0.
    \end{equation*}
\end{itemize}

\begin{remark}\label{r:mandorla4}
Even in this case the second coordinates of all points are nonpositive and hence the example complies with the additional constraint of \cref{r:mandorla,r:mandorla2}.
\end{remark}

\subsection{\texorpdfstring{$k\oneminM_{\sigma_{PT}}=0$}{Infinite min-max counterexample} in a metric space}
Let $\sigma:=\sigma_{PT}$, so that $8\sigma^3+4\sigma^2-3\sigma-3=0$.

We describe an explicit example that satisfies the requirements of \cref{thm:0.726-metric}.
Observe that this example is sharp. Indeed, in view of \cref{thm:01-minmax-metric}, such a construction could not work for a larger $\sigma$.

\subsubsection{Setup}
Let $(X, d)$ be the metric graph with four vertices $p_1, p_2, p_3, p_4$ connected by four edges to form a cycle. The length of the edges are (for the sake of brevity we denote $d_{ij}:= d(p_i, p_j)$)
\begin{equation*}\begin{aligned}
    d_{12} &:= \tfrac{2\sigma^2-\sigma+1}{\sigma+1} < 1,\\
    d_{23} &:= \tfrac{2\sigma}{\sigma+1} < 1,\\
    d_{34} &:= \tfrac{2\sigma^2}{\sigma+1} < 1, \\
    d_{41} &:= 1.
\end{aligned}\end{equation*}
Moreover, we add four infinite length edges starting from each of $p_1, p_2, p_3, p_4$. The distance $d$ on $X$ is the geodesic distance induced by the edges. In particular, one has
\begin{equation*}\begin{aligned}
    d_{13} &= d_{14} + d_{43} = d_{12} + d_{23} = \tfrac{2\sigma^2+\sigma+1}{\sigma+1} > 2\sigma, \\
    d_{24} &= \min\{d_{21} + d_{14},\, d_{23} + d_{34}\}
    = d_{23} + d_{34} = 2\sigma.
\end{aligned}\end{equation*}
Let $P=\{p_1, p_2, p_3, p_4\}$ and set $O:= p_1$. See \cref{fig:metric-quadrilateral}.


\begin{figure}
\newcommand\PointOnCircle[2]{
    \filldraw[black] (#1:1) circle (1pt);
    \node at (#1:0.85) {#2};
    \draw (#1:1) -- (#1:1.4);
    \draw[dashed] (#1:1.4) -- (#1:1.7);
}

\begin{tikzpicture}[scale=2]
\draw (0,0) circle [radius=1];

\PointOnCircle{180}{$\quad\quad\ p_1=O$}
\PointOnCircle{94}{$p_2$}
\PointOnCircle{0}{$p_3$}
\PointOnCircle{291.7}{$p_4$}

\node at (130:1.25) {$\frac{2\sigma^2-\sigma+1}{\sigma+1}$};
\node at (40:1.25) {$\frac{2\sigma}{\sigma+1}$};
\node at (-30:1.25) {$\frac{2\sigma^2}{\sigma+1}$};
\node at (235:1.2) {$1$};

\end{tikzpicture}
\caption{The metric space $(X, d)$ and the four-point set $P=\{p_1, p_2, p_3, p_4\}$ that proves $k\oneminM_{\sigma_{PT}}(X, d, O)=0$.}\label{fig:metric-quadrilateral}
\end{figure}
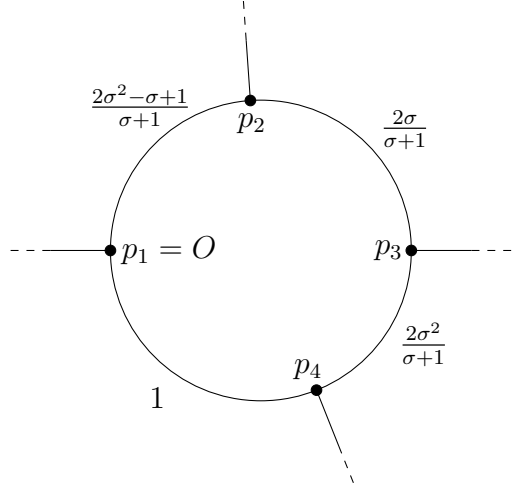

Since we know all the mutual distances between points in $P$, it is simple to check that
\begin{equation*}
    (p_2, p_4)\in\Delta_\sigma(p_1, 1),\ 
    (p_1, p_3)\in\Delta_\sigma(p_2, 1),\ 
    (p_2, p_4)\in\Delta_\sigma(p_3, 1),\ 
    (p_1, p_3)\in\Delta_\sigma(p_4, 1)
\end{equation*}
and therefore $P$ is $(\sigma,\infty)$-one-stable.

The decision to define $(X,d)$ as a metric graph (instead of just letting $X$ be equal to $P$) and to add rays emanating from $p_1, p_2, p_3, p_4$ guarantees that for any $r\in\mathcal R(P)$, the formulas \cref{eq:RU-formulas} for $R(P, r)$ and $\diam(U(P, r))$ hold.

\subsubsection{Proof of $M_\sigma(P)=0$}
It remains to prove that $M_\sigma(P)=0$. The fundamental observation --- as well as the main idea behind the construction --- is that the triangle $\{p_1, p_2, p_4\}$ yields the equality case of the computations for $M_\sigma^\sharp$ performed in \cref{sec:0.726}. In particular, applying \cref{e:max-diesis}, we have $M^\sharp_\sigma(\{p_1, p_2, p_4\})= 0$.  

As we have already observed in the previous two subsections, it is sufficient to show that $M_\sigma^\sharp(P)=0$ and $\bar M_\sigma^\flat(P)\leq 0$ (where $\bar M_\sigma^\flat$ is defined as in \cref{eq:emmebarrata}).
With the help of a computer one can check that $\bar M_\sigma^\flat(P) = -0.148\ldots < 0$, but in fact, following arguments entirely analogous to the ones which we now give, the inequality $\bar M_\sigma^\flat (P)<0$ can be proved with ``pen and paper'' with a little additional effort.

Now, we show that $F_\sigma^\sharp(P, r)\le M^\sharp_\sigma(\{p_1, p_2, p_4\})= 0$ for all $r\in\mathcal R(P)$. For the sake of notational simplicity we denote $r_i:=r(p_i)$.
Our strategy is to show that it is never convenient to have $r_3>0$.

If $r_3=0$, then $F_\sigma^\sharp(P, r) = F_\sigma^\sharp(\{p_1, p_2, p_4\}, r) \le M_\sigma^\sharp(\{p_1, p_2, p_4\}) = 0$. 

On the other hand, if $r_3>0$, since $B_{r_2}(p_2)$ and $B_{r_4}(p_4)$ cannot contain $p_3$, we have $R(P, r) = d_{13}+r_3$ and therefore
\begin{equation}\label{eq:Fsharp-pos3}
    2\sigma F^\sharp_\sigma(P, r) = 2\sigma(r_1+r_2+r_3+r_4) - (\tfrac12 + d_{13} + r_3).
\end{equation}

We want to use, as competitor for $r$, the family of radii 
\begin{equation*}
    (r'_1, r'_2, r'_3, r'_4) := (\max\{0, r_1-r_3\}, r_2 + r_3, 0, r_4 + r_3).
\end{equation*}
Let us check that $r'\in\mathcal R(P)$.
We have $r'_2=r_2+r_3\le d_{23}\le 1$ and $r'_4=r_4+r_3\le d_{34}\le 1$ and $r'_2+r'_4 = \le d_{23} + d_{34} = d_{24}$. 
If $r'_1=0$, there is nothing else to be checked to ensure that $r'\in\mathcal R(P)$. 
If $r'_1>0$, then it must be $r'_1=r_1-r_3$ and then we verify also that $r'_1+r'_2 = r_1+r_2\le d_{12}$ and analogously $r'_1+r'_4=r_1+r_4\le d_{14}$.
Observe also that 
\begin{equation}\begin{aligned}\label{eq:Fsharp-pos3-R}
    R(P, r') &= \max\{d_{12}+r'_2, d_{14}+r'_4\} = 
    \max\{d_{12}+r_2+r_3, d_{14}+r_4+r_3\}
    \\
    &\le
    \max\{d_{12}+d_{23}, d_{14}+d_{43}\}
    =
    d_{13}.
\end{aligned}\end{equation}

Thanks to \cref{eq:Fsharp-pos3,eq:Fsharp-pos3-R}, we can show that $r'$ is a better competitor than $r$:
\begin{equation*}\begin{aligned}
    2\sigma F^\sharp_\sigma(P, r) &=
    2\sigma(r_1+r_2+r_3+r_4) - (\tfrac12 + d_{13} + r_3) \\
    &\le
    2\sigma(r'_1+r'_2+r'_3+r'_4) - (\tfrac12 + d_{13} + r'_3)\\
    &\leq  
    2\sigma(r'_1+r'_2+r'_3+r'_4) - (\tfrac12 + R(P, r'))
     \\
    &=
    2\sigma F^\sharp_\sigma(P, r') = 2\sigma F^\sharp_\sigma(\{p_1, p_2, p_4\}, r')
    \le 2\sigma M^\sharp_\sigma(\{p_1, p_2, p_4\}) = 0.
\end{aligned}\end{equation*} 

\section{Computer-assisted proof of \texorpdfstring{\cref{thm:0.7}}{the 0.7 bound}}\label{s:0.7}
Let $\sigma:= 0.7$.
In this section we describe how to produce a computer-assisted proof of
\begin{equation}\label{eq:2-min-max-0.7-positive}
    2\oneminM_\sigma > 0.
\end{equation}
By definition, $2\oneminM_\sigma$ is the minimum of $M_\sigma(P)$ where $P$ runs over sets in $\mathcal F_\sigma(2)$.
Fix now $P\in \mathcal{F}_2 (\sigma)$. 
By construction $P=\{O, p_1, p_2, p_3, \ldots, p_8\}$ so that
$(p_1, p_2)\in \Delta_\sigma(O, 1)$, 
$(p_3, p_4)\in \Delta_\sigma(O, 1)$, 
$(p_5, p_6)\in \Delta_\sigma(p_1, 1)$, 
and $(p_7, p_8) \in \Delta_\sigma(p_2, 1)$. 
However, note that, for any such $P$, certainly also $P' := (O, p_1, p_2, p_5, p_6, p_7, p_8)$ belongs to $\mathcal{F}_\sigma(2)$ and moreover $M_\sigma(P') \le M_\sigma(P)$. In particular, we can rewrite the definition of $2\oneminM_\sigma$ in a simpler way, using $7$-point sets, which for ease of notation from now on we write as,
$\{O, p_1, p_2, q_{11}, q_{12}, q_{21}, q_{22}\}= \{O\}\cup P\cup Q$
\begin{equation*}
    2\oneminM_\sigma := \inf\bigg\{ M_\sigma(\{O\}\cup P\cup Q) : 
    \begin{array}{l}
        (p_1,p_2) \in \Delta_\sigma(O, 1) ,\,\\
        (q_{i1}, q_{i2})\in \Delta_\sigma(p_i, 1) \text{ for $i\in\{1,2\}$}
    \end{array}\bigg\}.
\end{equation*}
We settle for a possibly worse estimate but a less complicated problem by replacing $M_\sigma$ with $M_\sigma^\sharp$ (see \cref{eq:max_M_sharp_flat}), so that
\begin{equation*}
    2\oneminM_\sigma \ge \inf\bigg\{ M^\sharp_\sigma(\{O\}\cup P\cup Q) : 
    \begin{array}{l}
         (p_1,p_2) \in \Delta_\sigma(O, 1) ,\,\\
        (q_{i1}, q_{i2})\in \Delta_\sigma(p_i, 1) \text{ for $i\in\{1,2\}$}\\
    \end{array}\bigg\}.
\end{equation*}
Thus, to prove \cref{eq:2-min-max-0.7-positive}, it is sufficient to show that the right-hand side of the latter inequality is strictly positive.

Notice that the proof of \cref{lem:M-lip} shows also that, for any $(p_i)_{1\le i\le k}, (p'_i)_{1\le i\le k}\subseteq\R^2$,
\begin{equation}\label{eq:Msharp-lip}
    |M^\sharp_\sigma(\{p_1, p_2, \dots, p_k\})-M_\sigma^\sharp(\{p'_1, p'_2, \dots, p'_k\})| \le \sum_{i=1}^k |p_i-p'_i|.
\end{equation}
Already at this point, we have reduced our problem to showing that the Lipschitz function $M^\sharp_\sigma$ is positive over a certain compact subset of $\R^{12}$.
While this could theoretically be proven with a finite amount of computation (by properly partitioning the domain into smaller domains and using the Lipschitz continuity of $M^\sharp_\sigma$ to obtain lower bounds for $M^\sharp_\sigma$ on each smaller domain), such a task is out of reach without further simplifications. To convince oneself of the unfeasibility, observe that if we partition each coordinate in $50$ intervals (which is likely to be insufficient for our purposes) we would end up with $50^{12}\approx 2.5\cdot 10^{20}$ smaller domains and for each of them we would have to compute $M^\sharp_\sigma$. Assuming that a single call to $M^\sharp_\sigma$ requires\footnote{On a modern laptop our implementation requires $30$ microseconds. Thus our estimate is optimistic.} $1$ microsecond, the total computation would require three million years.

To proceed, let us begin by observing that, by rotational invariance, we may assume that $p_1=(0, y)$ for some $0\le y\le 1$. This observation reduces the dimension of the domain from $12$ to $11$.
Furthermore, by symmetry of the roles $p_1, p_2$ we can assume without loss of generality that $|p_1|\ge|p_2|$. Lastly, we may also assume that the first coordinate of $p_2$ is nonnegative.
These assumptions are contained in the following definition.

\begin{definition}
    Let $\Omega_\sigma \subseteq \R^2\times\R^2$ be the set of ordered pairs $(p_1, p_2)\in\R^2\times\R^2$ such that $(p_1,p_2)\in\Delta_\sigma(O, 1)$ with $p_1 = (0,y_1)$ for some $0\le y_1\le 1$, $p_2=(x_2, y_2)$ with $x_2\ge 0$, and $|p_1|\ge |p_2|$.
\end{definition}
Observe that $\Omega_\sigma$ is a $3$-dimensional bounded subset of $\R^2\times\R^2$.

\begin{definition}
    Given $p_1, p_2\in\R^2$, $r>0$, and $m\in\R$, let 
    \begin{equation*}
        X_\sigma(p_1, p_2, r, m) := \{q\in \overline{B_r(p_1)}:\, M_\sigma^\sharp(\{O, p_1, p_2, q\}) \le m\}.
    \end{equation*}
\end{definition}
Notice that in the definition of $X_\sigma(p_1, p_2, r, m)$ the order of $p_1$ and $p_2$ matters.

Fix\footnote{It would be more natural to fix $(p_1, p_2)\in\Omega_\sigma$. We do not force them to belong to $\Omega_\sigma$ because later on we will have to discretize the space and the pair $(p_1, p_2)$ may end up being very close to but not exactly in $\Omega_\sigma$.} $p_1, p_2 \in \R^2$ and consider the following sequence of increasingly weaker statements:
\begin{enumerate}[itemsep=4pt]
    \item[$(S.1)_\sigma$] $M_\sigma^\sharp(\{O, p_1, p_2\})>0$,
    \item[$(S.2)_\sigma$] The diameter of $X_\sigma(p_2, p_1, 1, 0)$ is $<2\sigma$.
    \item[$(S.3)_\sigma$] There is no choice of four points $q_{11}, q_{12}, q_{21}, q_{22}\in\R^2$ such that:
    \begin{itemize}
        \item $q_{11}, q_{12}\in X_\sigma(p_1, p_2, 1, 0)$;
        \item $q_{21}, q_{22}\in X_\sigma(p_2, p_1, 1, 0)$;
        \item $|q_{11}-q_{12}| \ge 2\sigma$;
        \item $|q_{21}-q_{22}| \ge 2\sigma$;
        \item $M_\sigma^\sharp(\{O, p_1, p_2, a, b\})\le 0$ for all $\{a, b\} \subset \{q_{11}, q_{12}, q_{21}, q_{22}\}$.
    \end{itemize}
    \item[$(S.4)_\sigma$] For any four points $q_{11}, q_{12}, q_{21}, q_{22}$ satisfying the first four conditions of the previous statement we have $M_\sigma^\sharp(\{O\}\cup P\cup Q) > 0$.
\end{enumerate}
As will be clear later on, to produce an efficient algorithm, it is crucial that $(S.3)_\sigma$ involves at most two points out of $\{q_{11}, q_{12}, q_{21}, q_{22}\}$ for each of the constraints it states. This will allow us to reduce substantially the dimension of the domain we are interested in.

\begin{lemma}\label{lem:algo_statement_implications}
    Fix $p_1, p_2\in \R^2$. Then
    $(S.1)_\sigma \implies (S.2)_\sigma \implies (S.3)_\sigma \implies (S.4)_\sigma$.
\end{lemma}
\begin{proof}
    If $(S.1)_\sigma$ holds, then $M_\sigma^\sharp(\{O, p_1, p_2, q\})\ge M_\sigma^\sharp(\{O, p_1, p_2\})>0$ for all $q\in\R^2$ and thus $X_\sigma(p_2, p_1, 1, 0)$ is empty. Hence $(S.2)_\sigma$ holds.

    If $(S.2)_\sigma$ holds, then it is impossible to find $q_{21}, q_{22}\in X_\sigma(p_2, p_1, 1, 0)$ such that $|q_{21}-q_{22}| \ge 2\sigma$ and thus $(S.3)_\sigma$ holds.

    If $(S.3)_\sigma$ holds, for any four points $q_{11}, q_{12}, q_{21}, q_{22}$ satisfying the first four conditions of $(S.3)_\sigma$ there exist $\{a, b\}\subset \{q_{11}, q_{12}, q_{21}, q_{22}\}$ such that $M_\sigma^\sharp(\{O, p_1, p_2, a, b\}) > 0$. Therefore,
    \begin{equation*}
        M_\sigma^\sharp(\{O,p_1,p_2,q_{11},q_{12},q_{21},q_{22}\}) 
        \ge
        M_\sigma^\sharp(\{O, p_1, p_2, a, b\}) > 0
    \end{equation*}
    as desired in $(S.4)_\sigma$.
\end{proof}

We have already observed that in order to prove \cref{eq:2-min-max-0.7-positive} it is sufficient to show the validity of $(S.4)_\sigma$ for all $(p_1, p_2)\in\Omega_\sigma$. To this aim, our (idealized) algorithm would iterate over all choices of $(p_1, p_2)\in\Omega_\sigma$ and perform the following sequence of checks:
\begin{enumerate}
    \item Check if $(S.1)_\sigma$ holds. If it holds, exit.
    \item Check if $(S.2)_\sigma$ holds. If it holds, exit.
    \item Check that $(S.3)_\sigma$ holds.
\end{enumerate}
Thanks to \cref{lem:algo_statement_implications}, this is sufficient to prove that $(S.4)_\sigma$ holds for all $(p_1, p_2)\in\Omega_\sigma$.

\subsection{Discretization of the first generation}
In order to transform the strategy described above into an effective algorithm to be executed on a machine, we adopt a discretization argument that reduces our scheme to a finite number of checks.
Let us now describe how we discretize the first generation of points (i.e., $p_1$ and $p_2$).

Given a positive real number $r>0$ and a point $x\in\R^d$, let $Q^d_r(x):=x + [-\tfrac r2,\tfrac r2]^d$.
Fix $\delta>0$; define $m^\delta := \sqrt2\delta$ and $r^\delta := 1 + \tfrac{\delta}{\sqrt2}$.

For a pair of points $(p_1,\, p_2)\in\R^2\times\R^2$, we consider the following sequence of statements:
\begin{enumerate}[itemsep=4pt]
    \item[$(S.1)_\sigma^\delta$] $M_\sigma^\sharp(\{O, p_1, p_2\})>m^\delta$,
    \item[$(S.2)_\sigma^\delta$] The diameter of $X_\sigma(p_2, p_1, r^\delta, m^\delta)$ is $<2\sigma$.
    \item[$(S.3)_\sigma^\delta$] There is no choice of four points $q_{11}, q_{12}, q_{21}, q_{22}\in\R^2$ such that:
    \begin{itemize}
        \item $q_{11}, q_{12}\in X_\sigma(p_1, p_2, r^\delta, m^\delta)$; 
        \item $q_{21}, q_{22}\in X_\sigma(p_2, p_1, r^\delta, m^\delta)$;
        \item $|q_{11}-q_{12}| \ge 2\sigma$;
        \item $|q_{21}-q_{22}| \ge 2\sigma$;
        \item $M_\sigma^\sharp(\{O, p_1, p_2, a, b\})\le m^\delta$ for all $\{a, b\} \subset \{q_{11}, q_{12}, q_{21}, q_{22}\}$.
    \end{itemize}
\end{enumerate}

We will repeatedly use the following simple observation.
If $(p'_1, p'_2)\in Q^4_\delta(p_1, p_2)$, then $|p'_1-p_1|, |p'_2-p_2| \le \frac\delta{\sqrt2}$. And in particular, $\max\{1+|p_1-p'_1|, 1+|p_2-p'_2|\}\le r^\delta$.

Let us show that the discretized statements $(S.1)_\sigma^\delta$, $(S.2)_\sigma^\delta$, $(S.3)_\sigma^\delta$ imply the validity of the original statements $(S.1)_\sigma$, $(S.2)_\sigma$, $(S.3)_\sigma$ in an open neighborhood.
\begin{lemma}\label{lem:first_discretization}
    Let $\delta>0$ be a positive real parameter and fix $(p_1, p_2)\in\R^2\times\R^2$. 
    Then the following statements hold.
    \begin{enumerate}[itemsep=4pt]
        \item If $(S.1)_\sigma^\delta$ holds for $(p_1, p_2)$, then $(S.1)_\sigma$ holds for all $(p'_1, p'_2)\in Q^4_\delta(p_1, p_2)$. 
        \item If $(S.2)_\sigma^{\delta}$ holds for $(p_1, p_2)$, then $(S.2)_\sigma$ holds for all $(p'_1, p'_2)\in Q^4_\delta(p_1, p_2)$.
        \item If $(S.3)_\sigma^{\delta}$ holds for $(p_1, p_2)$, then $(S.3)_\sigma$ holds for all $(p'_1, p'_2)\in Q^4_\delta(p_1, p_2)$.
    \end{enumerate}
\end{lemma}
\begin{proof}
    \begin{enumerate}
        \item By the Lipschitz continuity of $M_\sigma^\sharp$ (see \cref{eq:Msharp-lip}), assuming that $(S.1)_\sigma$ holds for $(p_1, p_2)$, we have
        \begin{equation*}
            M_\sigma^\sharp(\{O, p_1', p_2'\}) 
            \ge 
            M_\sigma^\sharp(\{O, p_1, p_2\}) 
            - \big(|p_1-p_1'| + |p_2-p_2'|\big)
            > m^\delta 
            - \big(\tfrac\delta{\sqrt 2} + \tfrac\delta{\sqrt 2}\big)
            = 0.
        \end{equation*}
        
        \item Take $q\in X_\sigma(p'_2, p'_1, 1, 0)$.
        Then $|q-p_2|\le |q-p'_2| + |p'_2-p_2| \le r^\delta$.
        Moreover, by \cref{eq:Msharp-lip}, we have
        \begin{equation*}
            M_\sigma^\sharp(\{O, p_1, p_2, q\}) 
            \le 
            M_\sigma^\sharp(\{O, p'_1, p'_2, q\})
            + \big(
            |p_1-p_1'| + |p_2-p_2'|
            \big)
            \le m^\delta.
        \end{equation*}
        Therefore, we have shows that $q\in X_\sigma(p_2, p_1, r^\delta, m^\delta)$. Since $q$ was arbitrary, we deduce that $X_\sigma(p'_2, p'_1, 1, 0)\subseteq X_\sigma(p_2, p_1, r^\delta, m^\delta)$. Then, the diameter of $X_\sigma(p'_2, p'_1, 1, 0)$ is smaller or equal to the diameter of $X_\sigma(p_2, p_1, r^\delta, m^\delta)$ and so the desired implication follows.
        
        \item Assume that $(S.3)_\sigma$ is false for $(p'_1, p'_2) \in Q_\delta^4(p_1, p_2)$. Let $q_{11}, q_{12}, q_{21}, q_{22}$ be the four points showing that $(S.3)_\sigma$ does not hold. We observed in the proof of statement (2) that $q_{11}, q_{12}\in X_\sigma(p_1, p_2, r^\delta, m^\delta)$ and $q_{21}, q_{22}\in X_\sigma(p_1, p_2, r^\delta, m^\delta)$. Furthermore, for any $\{a, b\}\subseteq\{q_{11}, q_{11}, q_{11}, q_{11}\}$,
        \begin{equation*}
            M_\sigma^\sharp(\{O, p_1, p_2, a, b\}) 
            \le M_\sigma^\sharp(\{O, p'_1, p'_2, a, b\}) + m^\delta 
            \le m^\delta.
        \end{equation*}
        Thus, we have discovered that also $(S.3)_\sigma^\delta$ is false for $(p_1, p_2)$, which concludes the proof.
    \end{enumerate}
\end{proof}

\subsection{Discretization of the second generation}
The statements $(S.2)_\sigma^\delta$ and $(S.3)_\sigma^\delta$ require some further discretization to become machine-verifiable. 

\begin{definition}
    Given $(p_1, p_2)\in\R^2\times\R^2$, $r>0$, $m\in\R$, and $\delta_1>0$, denote by $X_P^{\delta_1}(r, m)$ the discrete set
    \begin{equation*}
        X_\sigma^{\delta_1}(p_1, p_2, r, m) := 
            X_\sigma\Big(p_1, p_2, r+\tfrac{\delta_1}{\sqrt2}, 
            m + \tfrac{\delta_1}{\sqrt2}\Big) 
            \cap 
            \big(p_1+\delta_1\Z^2\big).
    \end{equation*}
\end{definition}
Let us begin by showing the decisive property of $X_\sigma^{\delta_1}(p_1, p_2, r, m)$.
\begin{lemma}\label{lem:X_net}
    The set $X_\sigma^{\delta_1}(p_1, p_2, r, m)$ is a $\tfrac{\delta_1}{\sqrt2}$-net for $X_\sigma(p_1, p_2, r, m)$.
\end{lemma}
\begin{proof}
    Fix an arbitrary point $q\in X_\sigma(p_1, p_2, r, m)$. We want to show that there is $q'\in X_\sigma^{\delta_1}(p_1, p_2, r, m)$ such that $|q-q'| \le \tfrac{\delta_1}{\sqrt2}$.

    Observe that:
    \begin{itemize}
        \item Since $q\in \overline{B_r(p_1)}$, then $\overline{B_{\tfrac{\delta_1}{\sqrt2}}(q)} \subseteq \overline{B_{r+\tfrac{\delta_1}{\sqrt2}}(p_1)}$.
        \item Since $M_\sigma^\sharp(\{O, p_1, p_2, q\})\le m$, by the Lipschitz continuity of $M_\sigma^\sharp$ (see \cref{eq:Msharp-lip}), we have $M_\sigma^\sharp(\{O, p_1, p_2, \tilde q\})\le m+\tfrac{\delta_1}{\sqrt2}$ for any $\tilde q\in \overline{B_{\tfrac{\delta_1}{\sqrt2}}(q)}$.
    \end{itemize}
    The two observations imply that $\overline{B_{\tfrac{\delta_1}{\sqrt2}}(q)}\subseteq X_\sigma(p_1, p_2, r+\tfrac{\delta_1}{\sqrt2}, m+\tfrac{\delta_1}{\sqrt2})$.
    
    Take $q'\in \overline{B_{\tfrac{\delta_1}{\sqrt2}}(q)}\cap(p_1+\delta_1\Z^2)$; thanks to what we have just shown we conclude that $q'\in  X_\sigma^{\delta_1}(p_1, p_2, r, m)$ as desired.
\end{proof}

Given $\delta_1>0$, consider the following statements
\begin{enumerate}[itemsep=4pt]
    \item[$(S.2)_\sigma^{\delta,\delta_1}$] The diameter of $X_\sigma^{\delta_1}(p_2, p_1, r^\delta, m^\delta)$ is $<2\sigma - \sqrt2\delta_1$.
    \item[$(S.3)_\sigma^{\delta,\delta_1}$] There is no choice of four points $q_{11}, q_{12}, q_{21}, q_{22}\in\R^2$ such that:
    \begin{itemize}
        \item $q_{11}, q_{12}\in 
        X_\sigma^{\delta_1}(p_1, p_2, 
            r^\delta, 
            m^\delta)$;
        \item $q_{21}, q_{22}\in 
        X_\sigma^{\delta_1}(p_2, p_1, 
            r^\delta, 
            m^\delta)$;
        \item $|q_{11}-q_{12}| \ge 2\sigma-\sqrt2\delta_1$;
        \item $|q_{21}-q_{22}| \ge 2\sigma-\sqrt2\delta_1$;
        \item $M_\sigma^\sharp(\{O, p_1, p_2, a, b\})\le m^\delta + \sqrt2\delta_1$ for all $\{a, b\} \subset \{q_{11}, q_{12}, q_{21}, q_{22}\}$.
    \end{itemize}
\end{enumerate}

We prove that the twice-discretized statements $(S.2)_\sigma^{\delta,\delta_1}$ and $(S.3)_\sigma^{\delta,\delta_1}$ imply the once-discretized versions $(S.2)_\sigma^\delta$ and $(S.3)_\sigma^\delta$.
\begin{lemma}\label{lem:second_discretization}
    Let $\delta>0$ and $\delta_1>0$ be two positive real parameters and fix $(p_1, p_2)\in\R^2\times\R^2$.
    Then, the following statements hold.
    \begin{enumerate}
        \setcounter{enumi}{1}
        \item If $(S.2)_\sigma^{\delta, \delta_1}$ holds for $(p_1, p_2)$, then also $(S.2)_\sigma^\delta$ holds for $(p_1, p_2)$.
        \item If $(S.3)_\sigma^{\delta, \delta_1}$ holds for $(p_1, p_2)$, then also $(S.3)_\sigma^\delta$ holds for $(p_1, p_2)$.
    \end{enumerate}
\end{lemma}
\begin{proof}
    \begin{enumerate}
        \setcounter{enumi}{1}
        \item Thanks to \cref{lem:X_net}, we know that $X_\sigma^{\delta_1}(p_2, p_1, r^\delta, m^\delta)$ is a $\tfrac{\delta_1}{\sqrt2}$-net for $X_\sigma(p_2, p_1, r^\delta, m^\delta)$. 
        Therefore, $\text{diam}(X_\sigma(p_2, p_1, r^\delta, m^\delta)) \le \text{diam}(X_\sigma^{\delta_1}(p_2, p_1, r^\delta, m^\delta)) + 2\tfrac{\delta_1}{\sqrt2}$ which proves the desired statement.
        \item Assume that $(S.3)_\sigma^\delta$ is false and let $q_{11}, q_{12}, q_{21}, q_{22}$ be the four points showing that $(S.3)_\sigma^\delta$ does not hold. 
        
        Thanks to \cref{lem:X_net}, we can find $q'_{11}, q'_{12}\in X_\sigma^{\delta_1}(p_1, p_2, r^\delta, m^\delta)$ and $q'_{21}, q'_{22}\in X_\sigma^{\delta_1}(p_2, p_1, r^\delta, m^\delta)$ such that $|q_{ij}-q'_{ij}| \le \tfrac{\delta_1}{\sqrt2}$ for all $i, j\in \{1, 2\}$. By definition of $q'_{11}, q'_{12}, q'_{21}, q'_{22}$, the first four conditions of $(S.3)_\sigma^{\delta,\delta_1}$ are satisfied by these four points. The fifth one follows from $|q_{ij}-q'_{ij}|\le \tfrac{\delta_1}{\sqrt2}$ together with the Lipschitz continuity of $M_\sigma^\sharp$ (see \cref{eq:Msharp-lip}).
    \end{enumerate}
\end{proof}

\subsection{Description of the algorithm}
We are now ready to describe in detail the algorithm we employ to prove \cref{eq:2-min-max-0.7-positive}.

Fix $\sigma:=0.7$ and $\delta:= 0.008$.
Let $\Omega'_\sigma\subseteq \R^2\times\R^2$ be a finite set such that 
\begin{equation*}
    \Omega_\sigma \subseteq \bigcup_{(p_1, p_2)\in \Omega'_\sigma} Q_\delta^4(p_1, p_2).
\end{equation*}
For each $(p_1, p_2)\in\Omega'_\sigma$:
\begin{enumerate}
    \item We try to show that $(S.1)_\sigma^\delta$ holds for $(p_1, p_2)$. If we fail, we go to the next step.
    \item We try to show that $(S.2)_\sigma^\delta$ holds for $(p_1, p_2)$. To this purpose, we try to show that $(S.2)_\sigma^{\delta,\delta_1}$ holds for at least one value of $\delta_1$ in $\{0.05, 0.02, 0.01\}$. If we fail, we go to the next step.
    \item We try to show that $(S.3)_\sigma^\delta$ holds for $(p_1, p_2)$.
    To this purpose, we try to show that $(S.3)_\sigma^{\delta,\delta_1}$ holds for $\delta_1:=0.03$. If we fail, we go to the next step.
    \item The algorithm fails.
\end{enumerate}

\begin{proposition}
    If the above-described algorithm terminates successfully then \cref{eq:2-min-max-0.7-positive} holds.
\end{proposition}
\begin{proof}
    The correct execution of the algorithm implies that, for all $(p_1, p_2)\in\Omega'_\sigma$, at least one out of $(S.1)_\sigma^{\delta}$, $(S.2)_\sigma^{\delta,0.05}$, $(S.2)_\sigma^{\delta,0.02}$, $(S.2)_\sigma^{\delta,0.01}$, $(S.3)_\sigma^{\delta,0.03}$ holds.
    Thanks to \cref{lem:second_discretization}, we deduce that at least one out of $(S.1)_\sigma^\delta$, $(S.2)_\sigma^\delta$, $(S.3)_\sigma^\delta$ holds for all $(p_1, p_2)\in\Omega'_\sigma$. 
    Then, applying \cref{lem:first_discretization}, due to the defining property of $\Omega'_\delta$, we deduce that at least one of $(S.1)_\sigma$, $(S.2)_\sigma$, $(S.3)_\sigma$ holds for all $(p_1, p_2)\in\Omega_\sigma$. 
    In particular, thanks to \cref{lem:algo_statement_implications}, we deduce that $(S.4)_\sigma$ holds for all $(p_1, p_2)\in\Omega_\sigma$. We have already explained why the validity of $(S.4)_\sigma$ at all points in $\Omega_\sigma$ implies \cref{eq:2-min-max-0.7-positive}.
\end{proof}

\subsection{Implementation}
A \texttt{C++} implementation of the algorithm outlined in the previous subsection is provided in \cref{app:source-code}.
The code prioritizes readability and clarity over speed optimization.

For each routine, the precise post-conditions are stated in the comment just before the routine. 
To properly grasp the post-conditions, it is essential to identify the floating-point numbers inputted and outputted with their corresponding rational numbers.

Let us draw a schematic comparison between the pseudo-algorithm and the actual source code:
\begin{itemize}
    \item The function \texttt{M\_sharp} computes $M_\sigma^\sharp$. It uses as subroutine \texttt{M\_plus\_sharp} which implements the function $M_\sigma^{+,\sharp}$ introduced in \cref{sec:linear-programming}. The function \texttt{M\_plus\_sharp} uses the open source library OR-Tools~\cite{ortools} to solve linear programming problems. The correctness of our algorithm does not depend on the correctness of such library as we do verify the correctness of the output produced by the library.
    \item The function \texttt{Omega\_discretized} returns the family of points that plays the role of $\Omega'_\sigma$.
    \item The function \texttt{X} returns the family of points that plays the role of $X^{\delta_1}_\sigma(p_1, p_2, r, m)$.
    \item The function \texttt{check\_S1} tries to prove $(S.1)_\sigma^\delta$.
    \item The function \texttt{check\_S2} tries to prove $(S.2)_\sigma^\delta$. More precisely, it tries to prove $(S.2)_\sigma^{\delta,\delta_1}$ for $\delta_1\in\{0.05, 0.02, 0.01\}$. It uses the subroutine \texttt{compute\_diameter} to compute the diameter of a set of points.
    \item The function \texttt{check\_S3} tries to prove $(S.3)_\sigma^\delta$.
    More precisely, it tries to prove $(S.3)_\sigma^{\delta,\delta_1}$ with $\delta_1:=0.03$. To this purpose, it constructs a graph whose vertices are the points $X_1\sqcup X_2 := X_\sigma^{\delta_1}(p_1,p_2, r^\delta, m^\delta) \sqcup X_\sigma^{\delta_1}(p_2,p_1, r^\delta, m^\delta)$. The edges are constructed according to the following rules:
    \begin{itemize}
        \item If $a, b\in X_i$, for $i=1,2$, then there is an edge between $a$ and $b$ if and only if $|a-b| \ge 2\sigma-\sqrt{2}\delta_1$ and $M_\sigma^\sharp(\{O, p_1, p_2, a, b\}) \le m^\delta+\sqrt{2}\delta_1$.
        \item If $a\in X_1$ and $b\in X_2$, then there is an edge between $a$ and $b$ if and only if $M_\sigma^\sharp(\{O, p_1, p_2, a, b\}) \le m^\delta+\sqrt{2}\delta_1$.
    \end{itemize}
    Observe that the property $(S.3)_\sigma^{\delta,\delta_1}$ is satisfied if and only if the given graph does not contain a $4$-clique (i.e., four vertices so that any two of them are connected by an edge) with two vertices in $X_1$ and two vertices in $X_2$.
    To check that such a special clique does not exist, the subroutine \texttt{contains\_bicolor\_k4} is called.
\end{itemize}

\subsection{Execution}
The code was compiled with:
\begin{center}
\texttt{g++ -std=c++20 -O2 -mfpmath=sse -msse2 -lortools besicovitch\_07.cpp -o besicovitch\_07}
\end{center}
The version of the compiler used was \texttt{gcc 10.3.1 20210422}.
The resulting executable file was run on a cluster of \texttt{quad 24 core 64-bit Intel Cascade Lake} processors.
The execution took approximately two hours.

Reproducing the execution on one core of a standard laptop would require approximately five days.

\subsection{Floating-point arithmetic issues}
In our implementation of the algorithm (available in \cref{app:source-code}), we extensively utilize floating-point arithmetic.
Floating-point arithmetic is notoriously challenging to use in computer-assisted proofs because of the inherent uncertainties (e.g., implementation-dependent behaviors) of the \emph{IEEE 754 standard}~\cite{IEEE}.
Recall that the fundamental \texttt{C++} type \texttt{double}, when compiled with \texttt{gcc}, adheres to the \texttt{binary64} format specified in the \emph{IEEE 754 standard}.\footnote{As explained at \url{https://gcc.gnu.org/wiki/FloatingPointMath}, the compliance to the standard --- at least for what concerns rounding --- is guaranteed for \texttt{double}s only with the compilation flags \texttt{-mfpmath=sse -msse2}.}
In particular the operations between \texttt{double}s comply with the following excerpt from the standard specification\cite[Section 5.1]{IEEE}:

\smallskip
\noindent\textit{[...] each of the computational
operations [...] that returns a numeric result shall be performed as if it first produced
an intermediate result correct to infinite precision [...] and then rounded that
intermediate result [...] to fit in the destination’s format [...]}
\smallskip

Let us collect some remarks about the use of floating-point numbers in our algorithm:
\begin{itemize}
\item Only basic operations on floating-point numbers are employed: addition, subtraction, multiplication, square root, and division. Moreover, for all instances of division, the denominator is always either $\sqrt{2}$ or $2\sigma=1.4$.

\item Throughout our computations, all floating-points numbers belong to the range $[-10, 10]$. Within this range, any real number can be represented by a valid floating-point number (of \texttt{double} type) with an absolute error smaller than $10^{-14}$ (see \cite[Table 3.2]{IEEE}).

\item All comparisons between floating-point values incorporate a tolerance threshold of $\texttt{EPS}=10^{-5}$ (i.e., if one shall prove $a<b$, then the code actually checks $a+\texttt{EPS}<b$).
\end{itemize}

Given the remarks above, it is a tedious but straightforward task to check that all the post-conditions stated in \cref{app:source-code} hold and do not depend on implementation-dependent properties of floating-point numbers.

\section{Quantitative Besicovitch conjecture}\label{s:quantitative}

In this final section of the paper, we propose the following conjecture, which we interpret as a quantitative version of Besicovitch's.

\begin{conjecture}\label{c:Bes-reg}
For any $\sigma > \frac{1}{2}$, there are positive constants $\eps$ and $\delta$ with the following property. 
Whenever $E\subset \mathbb R^2$ is a compact set such that
\begin{itemize}
\item[(a)] $\mathcal{H}^1 (E\cap U) \leq (1+\eps) \diam(U)$ for every $U$ with $\diam(U)\leq 1$, and
\item[(b)] $\mathcal{H}^1 (E\cap B_r (x))\geq 2 \sigma r$ for every $x\in E$ and every $r\leq 1$,
\end{itemize}
then any two connected components of $E$ are at least $\delta$ far apart. Therefore (provided $\eps$ is possibly smaller) $E$ is the union of finitely many disjoint embedded closed loops. 
\end{conjecture}
Let us remark that the conjecture is scaling-invariant, i.e., it would be equivalent to assume $\diam(U)\le r_0$ in (a), $r\le r_0$ in (b), and require two connected components to be $\delta r_0$ far apart in the conclusion.

We first explain the connection between \cref{c:Bes-reg} and \cref{c:Bes}. First of all fix a Borel set $E\subset \mathbb R^2$ with $\mathcal{H}^1 (E) < \infty$ and $\Theta_*^1 (E, \cdot) > \sigma$ $\mathcal{H}^1$-a.e. on $E$. Then for $\mathcal{H}^1$-a.e. $x\in E$ and for every fixed $\varepsilon>0$ there is a radius $r_0 (\varepsilon, x)$ such that
\begin{itemize}
\item[(a')] statement (a) above holds for every $U$ with ${\rm diam}\, (U) \leq 2 r_0 (\varepsilon, x)$ and $x\in U$;
\item[(b')] $\mathcal{H}^1 (B_r (x))\geq 2\sigma r$ for every $r \leq r_0 (\varepsilon, x)$.
\end{itemize}
The conditions (a) and (b) are therefore an obvious quantitative strengthening of the latter. Moreover, inspecting both the original arguments of \cite{Besicovitch} and those of \cite{PT} which are borrowed in \cref{sec:proof-infinite-minmax}, it is clear that (a') and (b') are the two key properties which ultimately underlie Besicovitch's, Preiss-Ti\v{s}er, and our upper bounds for $\bar\sigma$.

On the other hand the conclusion that $E$ is rectifiable can be equivalently rephrased as:
\begin{itemize}
\item[(R)] For $\mathcal{H}^1$-a.e. $x\in E$ and every $\varepsilon > 0$, there is a radius $r_0 (x, \varepsilon)>0$ and a closed connected set $F\subset B_{r_0} (x)$ with the property that $\mathcal{H}^1 ((E\Delta F)\cap B_{r_0} (x))<\varepsilon r_0$, where $E\Delta F$ is the symmetric difference of the two sets. 
\end{itemize}
Thus, clearly, the conclusion of \cref{c:Bes-reg} is a quantitative strengthening of (R). 

In this Section we just point out the following two facts, which are counterparts of corresponding statements for Besicovitch's \cref{c:Bes} but are, on the other hand, easier to prove.

\begin{proposition}\label{p:Bes-reg-2}
Assume that $\InfminM_{\sigma'} > 0$ for some $\sigma' \geq \frac{1}{2}$. Then \cref{c:Bes-reg} holds for $\sigma > \sigma'$.
\end{proposition}

\begin{proposition}\label{p:Bes-reg-1}
There is a compact set $E\subset \mathbb R \subset \mathbb R^2$ which has infinitely many connected components but satisfies the assumptions of \cref{c:Bes-reg} with $\sigma=\frac{1}{2}$ and $\varepsilon=0$. In particular for every positive number $\eta$ there is a pair of connected components of $E$ containing points at distance smaller than $\eta$.
\end{proposition}

\subsection{Proof of \texorpdfstring{\cref{p:Bes-reg-2}}{Proposition \ref{p:Bes-reg-2}}}
It suffices to show that, if $\InfminM_{\sigma'} (L) > 0$ for some $\sigma'> \frac{1}{2}$ and some $L>0$, then \cref{c:Bes-reg} holds for $\sigma>\sigma'$. We thus assume that we have a counterexample to the conjecture, which we denote by $E$, namely a set $E$ for which there are two connected components at distance $\delta$ but satisfies both (a) and (b) with $\eps>0$ (if the lower bound on the distance between any pair of connected components holds, then the last statement of the conjecture is a simple consequence, which we leave as an exercise to the reader). Both $\eps$ and $\delta$ can be chosen as small as we wish and their choice will be specified along the way at a suitable time in order to show a contradiction.

If the assertion is false there are two points $x_1$ and $x_2$ with $|x_1-x_2|\leq \delta$ belonging to distinct connected components. Let $E_1$ be a (relatively) closed and open set containing $x_1$ but not $x_2$ and let $E_2:=E\setminus E_1$.
Thus, $E_1$ and $E_2$ are two compact sets at a positive distance smaller than $\delta$ such that $E_1\sqcup E_2 = E$.
Let $x_1\in E_1$ and $x_2\in E_2$ be two points such that $|x_1-x_2|= \dist(E_1, E_2)$. 
We now apply a homothetic rescaling to $E$ so that $|x_1-x_2|=1$. In particular we conclude that 
\begin{equation}\label{e:aggiunta_Davide_1}
\mathcal{H}^1 (U\cap E) \leq (1+\eps) \diam(U)
\end{equation}
whenever $\diam(U) \leq 2L+3$, provided that $\delta^{-1}$ is larger than $2L+4$. Furthermore 
\[
\mathcal{H}^1 (B_r (x) \cap E) \geq 2\sigma r \qquad \forall r\leq 1, \forall x\in E\, .
\]

We can now apply \cref{l:furbata} to $\mu = (1+\varepsilon)^{-1} \mathcal{H}^1 \res (E\cap B_{L+2} (\frac{x_1+x_2}{2}))$ and conclude that, upon relabeling the sets and translating so that $x_1=O$, we have
\begin{equation}\label{e:aggiunta_Davide_2}
\mathcal{H}^1 (E_1 \cap B_R (O)) \leq (1+\varepsilon) (R+\tfrac{1}{2}) \qquad \forall R\leq L+1\, .
\end{equation}
We observe that $E_1$ is then $(\sigma', L)$-stable, provided $\eps$ is small enough. In fact, if $x\in E_1$ and $r\leq 1$, $\mathcal{H}^1 (B_r (x)\cap E_1)= \mathcal{H}^1 (B_r (x) \cap E)\geq 2\sigma r$ (because $B_1 (x) \cap E_2 = (B_r (x)\cap E)\setminus E_1 = \emptyset$) and $\mathcal{H}^1 (E_1\cap B_r (x))\leq (1+\varepsilon) \diam(E_1\cap B_r (x))$. In particular
\[
\diam(E_1 \cap B_r (x)) \geq \frac{2\sigma}{1+\varepsilon} r \geq 2\sigma' r \, ,
\]
provided $\sigma' (1+\varepsilon) \leq \sigma$. This ensures the existence of two points $q_1, q_2\in \overline{B_r(x)}$ which belong to $E_1$ so that $|q_1-q_2|\ge 2\sigma'r$, and hence the $(\sigma',L)$-stability condition. 

Consider now a set of points $P\subseteq E_1\cap B_L (O)$ and a family of radii $r\in\mathcal R(P)$. Recall the definitions of $U(P, r)$ and $R(P, r)$ given in \cref{subsec:objective}. 
Observe that 
\[
\mathcal{H}^1 (U(P,r)\cap E_1) = \sum_{p\in P} \mathcal H^1(B_{r(p)}(p)\cap E_1) \geq 2\sigma \sum_{p\in P} r(p)
\]
and, by \eqref{e:aggiunta_Davide_1} and \eqref{e:aggiunta_Davide_2},
\[
\mathcal{H}^1 (U(P,r)\cap E_1) \leq (1+\varepsilon) \min \{ \diam(U (P, r)), R (P, r) + \tfrac{1}{2}\}\, .
\]
These two inequalities imply 
\[
F_{\sigma'}(P, r) \le F_{\frac{\sigma}{1+\eps}}(P,r) = 
\sum_{p\in P} r(p) - \frac{1+\eps}{2\sigma} \min \{ {\rm diam}\, (U (P, r)), R (P, r) + \tfrac{1}{2}\} \le 0.
\]
Since $P\subseteq E_1\cap B_L(O)$ and $r\in \mathcal R(P)$ are arbitrary, we deduce that $M_{\sigma'}(E_1) = 0$ and so $\InfminM_{\sigma'}(L)=0$ which is the desired contradiction.

\subsection{Proof of \texorpdfstring{\cref{p:Bes-reg-1}}{Proposition \ref{p:Bes-reg-1}}}
For every integer $k\geq -1$ we let $I_k$ be the closed intervals 
\begin{align}
I_k &= [2^{-k}, \tfrac{7}{4}\cdot 2^{-k}] \qquad \mbox{for $k\geq 0$}\\
I_{-1} &= [-1, 0]\, .
\end{align}
The set $E$ is then given by $E := \bigcup_{k\geq -1} I_k$,
see \cref{f:esempio}.

\begin{figure}
\begin{tikzpicture}
\foreach \i in {0,...,10}
{
\draw[very thick] ({3*2^(-\i)},0) -- ({3*7*2^(-2-\i)},0);
}
\node[below] at (0,0) {$0$};
\foreach \i in {0,...,3}
{\node[above] at ({3*11/8*2^(-\i)},0) {$I_{\i}$};
}
\draw[very thick] (-1,0) -- (0,0);
\node[above] at (-0.5,0) {$I_{-1}$};
\end{tikzpicture}
\caption{The intervals $I_k$ on the real axis: the set $E$ is the union of them.}\label{f:esempio}
\end{figure} 

Obviously $E$ is closed, it has infinitely many connected components and $\mathcal{H}^1 (E) < \infty$. 
We next wish to show that there is $r_0>0$ such that 
\begin{equation}\label{e:lb-intervals}
\mathcal{H}^1 (E\cap [a-r, a+r]) \geq \mathcal{H}^1 ([a-r,a+r]\setminus E)\qquad \forall r<r_0
\end{equation}
and for all $a\in E$, which can be easily seen to complete the proof of the proposition. It is also easy to see that, if we focus on points $a$ belonging to $I_0\cup I_{-1}$, then there is such an $r_0$. We hence assume that $a\in I_k$ for some $k\geq 1$. The inequality is then a direct consequence of the following observation. Fix any interval $A$ among the intervals $I_k$ with $k\geq 1$ (in other words $A$ is neither the leftmost, nor the rightmost, interval of the collection). Denote by $B$ and $C$ the open intervals which form the two connected components of $\mathbb R\setminus E$ adjacent to $A$. Then 
\begin{equation}\label{e:left-right}
 \mathcal{H}^1 (A) \geq \mathcal{H}^1 (B) + \mathcal{H}^1 (C)\, 
\end{equation}
(the number on the left is $(1-4^{-1}) 2^{-k}$ while the one on the right is $4^{-1} (2^{-k-1}+2^{-k})$). 
Armed with the latter observation we show that \eqref{e:lb-intervals} holds for any $a\in I_k$ with $k\geq 1$. 

Fix $a$ and $r\leq r_0$ (whose choice will be specified at the end) and introduce $b:= \max \{a-r, 0\}$. If $r_0\leq 1$ we are guaranteed that, when $a-r<0$, then $[a-r, 0]$ is contained in $E$. So, we just need to show 
\begin{equation}\label{e:lb-intervals-2}
\mathcal{H}^1 (E\cap [b, a+r]) \geq \mathcal{H}^1 ([b,a+r]\setminus E)\qquad \forall r<r_0\,.
\end{equation}
Next consider the case when $E\cap [b, a+r]$ does not include any $I_j$. Then necessarily $b=a-r$ and, if we let $k$ be such that $I_k\ni a$, either $a+r\in I_k$, or $a-r\in I_k$, namely either $[a-r, a]\subset I_k$ or $[a, a+r]\subset I_k$ (or both!), which makes \eqref{e:lb-intervals-2} a triviality.

We assume therefore that among the intervals forming $E\cap [b, a+r]$ there is at least one entire $I_j$. This assumption guarantees that, for each connected component $B$ of $\R\setminus E$ that intersects $[b, a+r]$, at least one of the connected components of $E$ adjacent to $B$ is entirely contained in $[b, a+r]$. By choosing $r_0>0$ small enough so that $I_0$ is never entirely contained in $[b, a+r_0]\supset [b, a+r]$, the previous observation together with \cref{e:left-right} implies \cref{e:lb-intervals-2}.


\appendix
\clearpage
\section{Source code of the computer-assisted proof of \texorpdfstring{$\bar\sigma\le 0.7$}{sigma<=0.7}}\label{app:source-code}

\definecolor{LightGray}{gray}{0.97}
\begin{minted}[frame=lines,
framesep=2mm,
baselinestretch=1.2,
bgcolor=LightGray,
fontsize=\footnotesize,
linenos]{C++}
#include <bits/stdc++.h>
#include "ortools/linear_solver/linear_solver.h"
using namespace std;
using namespace operations_research;

const double sigma = 0.7;

// We use an EPS of tolerance everywhere to avoid precision issues due 
// to floating-point arithmetic. 
// All operations are numerically stable.
const static double EPS = 1e-5;

// Returns the time elapsed in nanoseconds from 1 January 1970, at 00:00:00.
long long get_time() {
    return chrono::duration_cast<chrono::nanoseconds>(
        chrono::steady_clock::now().time_since_epoch())
        .count();
}


// Class that represents a point in the plane.
struct point {
    double x,y;

    point() : x(0), y(0) {}
    point(double x, double y) : x(x), y(y) {}

    point operator -(const point& other) const {
        return point(x-other.x, y-other.y);
    }
    point operator +(const point& other) const {
        return point(x + other.x, y + other.y);
    }
};

const point O = {0, 0};

double norm(point A) { return sqrt(A.x*A.x + A.y*A.y); }

ostream& operator <<(ostream& out, point P) {
    out << "(" << P.x << ", " << P.y << ")";
    return out;
}


// Returns M^{+,\sharp}_\sigma(P, P[i0])
//
// Let res be the value returned by this function.
// We guarantee only that 
//   res <= M^{\sharp}_\sigma(Q). 
// Notice that the index i0 is dropped.
//
// Even though we use the library ortools, the validity of the 
// post-condition stated above does not depend on the correctness
// of the library. We verify at the end of this function that the 
// condition holds by producing a family r\in \mathcal R(P) such that
// F^{\sharp}_\sigma(P, r) = res. 
double M_plus_sharp(const vector<point>& P, int i0) {
    int k = P.size();
    
    unique_ptr<MPSolver> solver(MPSolver::CreateSolver("GLOP"));

    
    // Variables:
    // R[0,...,k-1]
    const double infinity = solver->infinity();
    vector<MPVariable*> r;
    for (int i = 0; i < k; i++) 
        r.push_back(solver->MakeNumVar(0.0, 1.0, "r[" + to_string(i) + "]"));
    
    // norm(P[i]) + r[i] <= norm(P[i0]) + r[i0]
    for (int i = 0; i < k; i++) {
        if (i == i0) continue;
        MPConstraint* const c = solver->MakeRowConstraint(
            -infinity, norm(P[i0])-norm(P[i]));
        c->SetCoefficient(r[i], 1);
        c->SetCoefficient(r[i0], -1);
        
    }
    // r[i] + r[j] <= norm(P[i] - P[j])
    for (int i = 0; i < k; i++) for (int j = i+1; j < k; j++) {
        MPConstraint* const c = solver->MakeRowConstraint(
            -infinity, norm(P[i]-P[j]));
        c->SetCoefficient(r[i], 1.0);
        c->SetCoefficient(r[j], 1.0);
    }
    // Goal:
    //  max (r[0] + r[1] + ... + r[k-1]) - r[i0] / (2*sigma)
    MPObjective* const objective = solver->MutableObjective();
    for (int i = 0; i < k; i++) objective->SetCoefficient(r[i], 1);
    // The next line overwrites the value set in the previous line.
    objective->SetCoefficient(r[i0], 1 - 1.0 / (2*sigma));
    objective->SetMaximization();
    
    const MPSolver::ResultStatus result_status = solver->Solve();
    // If there is not a choice of r[0], ..., r[k-1] satisfying all
    // constraints, then we return -1.
    // Observe that r[0] = r[1] = ... = r[k-1] = 0 is not a solution
    // of the constraints norm(P[i]) + r[i] <= norm(P[i0]) + r[i0]
    // unless norm(Q[i0]) is the largest of norm(Q[i]).
    if (result_status != MPSolver::OPTIMAL) return -1;

    // Verification of the solution.
    vector<double> r_sol(k);
    for (int i = 0; i < k; i++)
        r_sol[i] = max(0.0, r[i]->solution_value() - 2*EPS);
    
    for (int i = 0; i < k; i++) 
        assert(0 <= r_sol[i] and r_sol[i] <= 1);
    for (int i = 0; i < k; i++) for (int j = i+1; j < k; j++) 
        assert(r_sol[i] + r_sol[j] <= max(0.0, norm(P[i] - P[j]) - EPS));
    
    // Computation of the solution value.
    double big_radius = 0.0;
    for (int i = 0; i < k; i++) 
        big_radius = max(big_radius, norm(P[i]) + r_sol[i]);
    
    double res = 0;
    for (int i = 0; i < k; i++)
        res += r_sol[i];
    res -= (big_radius + 0.5) / (2*sigma);
    
    return res - EPS;
}

// Returns M^\sharp_\sigma(P).
//
// Let res be the value returned by this function.
// We guarantee that 
//   res <= M^{\sharp}_\sigma(P).
//
// The parameter use_last_points forces the subset P' (that will be
// passed to M_plus_sharp) to contain the last use_last_points points 
// in P. 
// This parameter is used with values 1 or 2, to speed up the execution. 
// Observe that setting use_last_points to a positive value can only 
// decrease the value returned by this function.
// When we call this function with use_last_points > 0, we are able to
// prove that the result is not affected at all by such choice of
// use_last_points.
double M_sharp(vector<point> P, int use_last_points=0) {
    int k = P.size();
    
    double ans = -1;
    
    // The bitmask bb iterates over the subsets of {0, 1, ..., k-1}.
    for (int bb = 0; bb < (1<<k); bb++) {
        vector<point> P_prime;
        // For subsets of <= 1, M_plus_sharp returns a negative number.
        // Therefore we skip those subsets.
        if (__builtin_popcount(bb) <= 1) continue;
        
        // We want to consider only the subsets bb containing the
        // last use_last_points points.
        bool using_last_points = true;
        for (int i = k-1; i >= k-use_last_points; i--)
            using_last_points &= ((bb&(1<<i)) > 0);
        if (!using_last_points) continue;
        
        for (int i = 0; i < k; i++) 
            if (bb & (1<<i)) P_prime.push_back(P[i]);
        
        for (int i = 0; i < ssize(P_prime); i++) 
            ans = max(ans, M_plus_sharp(P_prime, i));
    }

    return ans;
}


// Returns the maximum distance between two of the given points.
//
// Let res be the value returned by this function.
// We guarantee that 
//   res >= diam(P).
double compute_diameter(const vector<point>& P) {
    int n = ssize(P);
    double diam = 0;
    for (int i = 0; i < n; i++) 
        for (int j = i+1; j < n; j++) 
            diam = max(diam, norm(P[i]-P[j]));
    return diam + EPS;
}


// Returns true if there are four vertices forming a K_4 with two 
// vertices of color 0 and two vertices of color 1.
bool contains_bicolor_k4(int n, vector<bool> color, vector<set<int>> edges) {
    for (int i = 0; i < n; i++) {
        for (int j = i + 1; j < n; j++) {
            if (color[i] or color[j]) continue;
            if (!edges[i].count(j)) continue;
            // i and j have color 0 and they know eachother.
            vector<int> candidates;
            for (int x: edges[i]) {
                if (color[x] == 1 and edges[j].count(x))
                    candidates.push_back(x);
            }
            for (int x: candidates) for (int y: candidates) {
                if (x >= y) continue;
                if (edges[x].count(y)) return true;
            }
        }
    }
    return false;
}


// Returns \Omega_\sigma \cap (\delta \{0}\times \Z^3).
// 
// Let P be the vector of pairs returned by this function.
// We guarantee only that the union of Q_\delta(p_1, p_2), for 
// (p_1, p_2) \in P, covers \Omega_\sigma.
vector<pair<point,point>> Omega_discretized(double delta) {
    // Observe that \Omega_\sigma is a subset of \{0\} \times [-1, 1]^3.
    const double l = 1 + 3*delta;
    
    point p_1 = O;
    point p_2 = O;
    vector<pair<point,point>> P;
    for (p_1.y = -l; p_1.y <= l; p_1.y += (delta-EPS)) 
    for (p_2.x = -l; p_2.x <= l; p_2.x += (delta-EPS))
    for (p_2.y = -l; p_2.y <= l; p_2.y += (delta-EPS)) {
        auto cmp = [&](double x, double y, double lip_error) { 
            return x <= y + lip_error + EPS; 
        };
        if (cmp(0, p_1.y, delta/2) 
            and cmp(p_1.y, 1, delta/2)
            and cmp(norm(p_2), norm(p_1), sqrt(2)*delta)
            and cmp(2*sigma, norm(p_1 - p_2), sqrt(2)*delta)
            and cmp(0, p_2.x, delta/2))
                P.push_back({p_1, p_2});
    }
    return P;
}


// Returns the set X_\sigma^{\delta_1}(p_1, p_2, r, m).
//
// Let P be the vector returned by this function.
// We guarantee only that P is a \delta_1/sqrt{2}-net for the set 
// X_\sigma(p_1, p_2, r, m).
vector<point> X(double delta_1, const point p_1, const point p_2,
                double r, double m) {
    delta_1 -= EPS;
    vector<point> P;
    r += delta_1/sqrt(2);
    
    for (double x = -r; x <= r; x += delta_1) {
        for (double y = -r; y <= r; y += delta_1) {
            point q = {x, y};
            if (norm(q) <= r and M_sharp({O, p_1, p_2, p_1 + q}) 
                                 <= m + delta_1 / sqrt(2) + EPS)
                P.push_back(p_1 + q);
        }
    }
    return P;
}



// If it returns true, then (S.1)_\sigma^\delta holds for (p_1, p_2).
bool check_S1(double delta, point p_1, point p_2) {
    const double m_delta = sqrt(2) * delta;
    return M_sharp({O, p_1, p_2}) > m_delta + EPS;
}


// If it returns true, then (S.2)_\sigma^\delta holds for (p_1, p_2).
bool check_S2(double delta, point p_1, point p_2) {
    double m_delta = sqrt(2)*delta;
    double r_delta = 1 + delta / sqrt(2);
    
    for (double delta_1: {0.05, 0.02, 0.01}) {
        // We check the validity of (S.2)_\sigma^{\delta, \delta_1}.
        vector<point> X_2 = X(delta_1, p_2, p_1, r_delta, m_delta);
        double diam = compute_diameter(X_2);
        if (EPS + diam < 2*sigma - sqrt(2)*delta_1) return true;
    }
    return false;
}

// If it returns true, then (S.3)_\sigma^\delta holds for (p_1, p_2).
bool check_S3(double delta, point p_1, point p_2) {
    double m_delta = sqrt(2) * delta;
    double r_delta = 1 + delta / sqrt(2);
        
    const double delta_1 = 0.03;
    // We check the validity of (S.3)_\sigma^{\delta,\delta_1}.
    // Choosing delta_1 = 0.05 would not sufficient.

    vector<point> X_1 = X(delta_1, p_1, p_2, r_delta, m_delta);
    vector<point> X_2 = X(delta_1, p_2, p_1, r_delta, m_delta);
    
    double dist = 2*sigma - sqrt(2) * delta_1;
    double m_val = m_delta + sqrt(2) * delta_1;
    function<bool(point,point)> M_compatible = [&](point X, point Y) {
        return M_sharp({O, p_1, p_2, X, Y}, 2) < m_val + EPS;
    };
    
    // We construct the graph that has X_1 and X_2 as vertices.
    // Two vertices are connected by an edge if they can be simultaneously
    // cousins while having M_sharp < 0 (taking perturbations into account).
    int n = ssize(X_1) + ssize(X_2);
    
    auto vertex = [&](int it) {
        assert(0 <= it < n);
        if (it < ssize(X_1)) return X_1[it];
        return X_2[it - ssize(X_1)];
    };
    
    
    vector<bool> color(n, false);
    for (int i = 0; i < ssize(X_1); i++) color[i] = true;
    vector<set<int>> edges(n);
    // First we find edges between vertices of the same color.
    for (int i = 0; i < n; i++) for (int j = i+1; j < n; j++) {
        if (color[i] == color[j]) {
            point a = vertex(i), b = vertex(j);
            if (norm(a - b) > dist - EPS and M_compatible(a, b)) {
                edges[i].insert(j);
                edges[j].insert(i);
            }
        }
    }
    // Then edges between vertices of different type.
    for (int i = 0; i < n; i++) for (int j = i+1; j < n; j++) {
        if (color[i] != color[j]) {
            point a = vertex(i), b = vertex(j);
            if (!edges[i].empty() and !edges[j].empty() 
                                  and M_compatible(a, b)) {
                edges[i].insert(j);
                edges[j].insert(i);
            }
        }
    }    
    // At this point, we show that there is not a bicolored K_4.
    return !contains_bicolor_k4(n, color, edges);
}

int main() {
    cout << "sigma = " << sigma << endl;
    const double delta = 0.008;
    
    // A set of points such that the cubes with side delta
    // centered at these points cover the space of parameters.
    // See the header of Omega_discretized for more details.
    auto Omega_prime = Omega_discretized(delta);
    
    // We shuffle the subregions to have reasonable statistics even 
    // after having processed only a small number of them.
    std::mt19937 g(2024);
    shuffle(Omega_prime.begin(), Omega_prime.end(), g);
    
    cout << "The interesting region of parameters is covered by "
         << ssize(Omega_prime)
         << " cubes of side "
         << delta << "." << endl << endl;
    
    long cnt = 0, cnt_1 = 0, cnt_2 = 0, cnt_3 = 0;
    long long initial_time = get_time();
    long long last_print_time = get_time();
    
    // For each of subregion, we prove that there is no counterexample 
    // in such subregion.
    for (auto p_12: Omega_prime) {
        // Print status (and some statistics).
        long long current_time = get_time();
        if (current_time - last_print_time > 1e9) {
            cout << "Processed " << cnt << " / " 
                 << ssize(Omega_prime) << " cubes." << endl;
            cout << "The distribution of the strategies: "
                 << cnt_1 << " " << cnt_2 << " " << cnt_3 << "." << endl;
            cout << "Average time per cube: " 
                 << (current_time - initial_time) / (cnt * 1e9)
                 << " seconds." << endl << endl;
            last_print_time = current_time;
        }
        cnt++;
        
        point p_1 = p_12.first;
        point p_2 = p_12.second;
        
        // We refer to the headers of these functions for an explanation.
        // If any of them returns true, then (S.4) holds for all points 
        // in in the region Q_\delta(p_1, p_2).
        if (check_S1(delta, p_1, p_2)) cnt_1++;
        else if (check_S2(delta, p_1, p_2)) cnt_2++;
        else if (check_S3(delta, p_1, p_2)) cnt_3++;
        else {
            cout << "We are not able to handle the cube with side "
                 << delta << " centered at: " << p_1 << ", " << p_2 
                 << "." << endl;
            assert(0);
            return 1;
        }
    }
}

\end{minted}

\printbibliography

\end{document}